\documentclass[11pt]{amsart}
\usepackage{amsthm, amsfonts, amsxtra, amssymb, amscd}

\usepackage{amsfonts,amsmath,amssymb,amscd,amsthm}
\usepackage[english]{babel} 
\usepackage[all]{xy}
 \usepackage[backref, colorlinks, linktocpage, citecolor = blue, linkcolor = blue]{hyperref}

\newcommand{\ii}{{\rm i} }

\newtheorem*{lemma*}{Lemma}
\newtheorem{lemma}[subsection]{Lemma}
\newtheorem*{theorem*}{Theorem}
\newtheorem{theorem}[subsection]{Theorem}
\newtheorem*{proposition*}{Proposition}
\newtheorem{proposition}[subsection]{Proposition}
\newtheorem*{corollary*}{Corollary}

\def\ch{Cayley--Hamilton}
\theoremstyle{definition}
\newtheorem*{definition*}{Definition}
\newtheorem{definition}[subsection]{Definition}
\newtheorem*{example*}{Example}
\newtheorem{example}[subsection]{Example}
\newtheorem{exercise}[subsection]{Exercise}

\theoremstyle{remark}
\newtheorem*{remark*}{Remark}
\newtheorem{remark}[subsection]{Remark}

\usepackage{multirow}
\newcommand{\C}{\mathbb C}
\newcommand{\N}{\mathbb N}
\newcommand{\Q}{\mathbb Q}

\newcommand{\Z}{\mathbb Z}\newcommand{\R}{\mathbb R}
\newcommand{\pd}[2]{\dfrac{\partial#1}{\partial#2}}

\renewcommand{\phi}{\varphi}

\newcommand{\be}{\begin{enumerate}}
\newcommand{\ee}{\end{enumerate}}

\title{Fermionic matrices and super Cayley--Hamilton algebras.}
\author{Claudio Procesi}
\begin{document}\address{Dipartimento di Matematica, G. Castelnuovo,
Universit\`a di Roma La Sapienza, piazzale A. Moro,  00185,
Roma, Italia}
\email{procesi@mat.uniroma1.it}\maketitle
 \begin{abstract}
In this paper we study invariants and trace identities of $m $--tuples of bosonic and fermionic $n\times n$ matrices.
\end{abstract}
\tableofcontents
\section{Introduction}
In a recent preprint: 

{\em 
Fermionic trace relations and supersymmetric
indices at finite N,} 

\noindent the authors 
Giorgos Eleftheriou, Ziming Ji, Sameer Murthy   start  in their words {\em to study invariants of bosonic and fermionic (Grassmann-valued) matrices
under the adjoint action of $U(N)$, weighted by the fermion number} \cite{zm}.\medskip

In this paper they discuss  Formulas for these invariants, which can be viewed as the content of  the {\em first fundamental theorem of matrix invariants} FFT, and  formulate two natural conjectures about trace relations, which can be viewed as the content of  the {\em second fundamental theorem of matrix invariants} SFT. We prove both in this paper. \medskip

Conjecture 1 is our Theorem  \ref{main1} while  Conjecture 2  is our Theorem \ref{trnm}.\smallskip

The proof of these theorems and also of the First Fundamental Theorem  \ref{FFT} follows the lines of my 1976 paper \cite{procesi}.\medskip

In fact the first fundamental theorem of matrix invariants  was already proved by Sibirskii \cite{sibirskii},  the  second fundamental theorem of matrix invariants  was independently proved by Razmyslov \cite{Raz}. \medskip

 Some literature of 3--dimensional invariant theory appears in papers on continuous mechanics, \cite{spencer00}, \cite{spencer01}, \cite{spencer1}, \cite{spencer2}, \cite{spencer3}, \medskip

As an application I define super \ch\ algebras and prove the strong  embedding theorem, Theorem  \ref{str}, for this class of graded algebras, following the lines of my 1987 paper \cite{P5}.\medskip

The  main  novelty in this paper is an appropriate generalization  of the symbolic  method and a very careful bookkeeping of the signs.
\medskip

{\em In this paper all formulas are written over the complex numbers $\C$ but really they should be written as formulas over the rational numbers $\Q$.}\medskip

There are several papers which treat a similar problem, but in the framework of matrix superalgebras.

Starting from work of Razmyslov \cite{Raz1}  and then of Berele \cite{ber0},  \cite{ber}, \cite{ber1}.

I will not discuss relationships between these results and the present paper.

\subsection{ One fermionic matrix}\label{unf} This case has been completely  described in \cite{P6}  and \cite{bps}. 

The problem is to describe the invariant algebra  $$(\bigwedge M_n(\C)^*\otimes M_n(\C))^G=\left( M_n(\bigwedge M_n(\C)^*)\right)^G$$  under conjugation by $G=GL(n,\C)$.\smallskip

The action  on $\bigwedge M_n(\C)^*\otimes M_n(\C)$  is the tensor product action, where the action  on $M_n(\C)$ is conjugation and that on  $\bigwedge M_n(\C)^*$ is the action, as algebra automorphisms, induced by the action on the dual $M_n(\C)^*$ of conjugation.\smallskip

The matrix algebra  $M_n(\C)$ has as basis the matrix units $e_i\otimes e^j$ and its dual $M_n(\C)^*$ the duals $\xi_{i,j}:=e^i\otimes e_j$ so that defining $ \xi:=\sum_{i,j}\xi_{i,j}e_i\otimes e^j$ we have a conjugation invariant {\em generic matrix} \label{genn} with entries $\xi_{i,j}$ in $\bigwedge M_n(\C)^*$, we write $ \xi=(\xi_{i,j})$.\medskip

One has $tr( \xi^{2k})=0,\ \forall k$ and $ \xi^{2n}=0$, see \cite{P6}, which is an interpretation of the classical Theorem of Amitsur--Levitzi on the vanishing of the standard identity in $2n$ variables \cite{AL}.\smallskip

The computation of  $(\bigwedge M_n(\C)^* )^G$ is a classical Theorem of Dynkin, see \cite{din}.   \medskip

 We have that $(\bigwedge M_n(\C)^* )^G=\bigwedge (t_1,\cdots,t_n),\quad t_i:=tr( \xi^{2i-1})$ is the exterior algebra in the $n$  Grassmann variables $t_i$ of degree $2i-1$.\smallskip
 
 Finally  $R:=(\bigwedge M_n(\C)^*\otimes M_n(\C))^G$ has been described in a paper by Bre\v{s}ar; Procesi; \v{S}penko  \cite{bps}.
 
 The invariant algebra $R$ is a free module with basis the $2n$  elements $ \xi^i,\ i=0,\cdots,2n-1$  over the subalgebra $\bigwedge (t_1,\cdots,t_{n-1})$ and, by Lemma 4.3 of \cite{bps} the {\em missing generator} $t_n$  is given by the relation:\begin{equation}\label{tru}
0= n \xi^{2n-1} -\sum_{
i=0} ^{n-1} \xi^{2i} \wedge t_{ n-i },\quad\text{or}\quad   t_n= n \xi^{2n-1}-\sum_{
i=1} ^{n-1} \xi^{2i} \wedge t_{ n-i  } .
\end{equation} In particular the graded dimension  polynomial is
 $$\prod_{i=0}^{n-1}(1+q^{2i-1})(\sum_{i=0}^{2n-1}(1+q^{i})).$$
 For several variables the results are not as precise, since they are not so precise even in the classical case. Nevertheless most of the classical results extend to this more general case.\smallskip
 
 The previous Theorem has then a vast generalization to simple Lie algebras by De Concini, Papi, Procesi \cite{dpp}.
 
\section{Grassmann matrices}\subsection{Superalgebras}
We use the language of superspaces and $\Z/(2)$  graded algebras or superalgebras.  
\begin{definition}\label{super} \strut
\begin{enumerate}\item A {\em superspace}  is a  $\Z/(2)$  graded space $W=W_0\oplus W_1$. 

\item A {\em superalgebra}  is an associative algebra $A$ with a  $\Z/(2)$  grading $A=A_0\oplus A_1$  so that $A_jA_j\subset A_{i+j},\  i+j \ \text{mod}\ 2$.
\item  We denote by $d(a)\in\{0,1\}$ the degree of a  homogeneous element $a$ of $A$.
\item The supercommutator of two homogeneous elements $a,b$ is $$[a,b]_s:=ab-(-1)^{d(a)d(b)}ba.$$ 
\item An algebra is supercommutative if the  supercommutator of two   elements  is always zero.

\item The  supercenter of $A$  is the set of elements of $A$ which super commute with all elements of $A$. It is a supercommutative subalgebra.
\end{enumerate}

\end{definition}

\smallskip

One has the notion of {\em graded tensor product} $A\widehat \otimes B$ of superspaces, which, as vector space, is $A \otimes B$ but inherits an obvious   $\Z/(2)$  grading from the grading of $A,B$.\smallskip

  When $A,B$ are both superalgebras this is also a superalgebra  which, as vector space, is $A \otimes B$ but for multiplication of graded elements one defines:
$$ a_1\otimes b_1\cdot a_2\otimes b_2=(-1)^{d(a_2)d(b_1)}a_1a_2\otimes b_1b_2.$$

If $A=A_0$ or $B=B_0$  the  multiplication in the graded tensor product equals that of the usual   tensor product.\medskip

\begin{proposition}\label{suc}
The graded tensor product $A\widehat \otimes B$ of two supercommutative superalgebras is supercommutative.
\end{proposition}
\begin{proof}
Exercise.
\end{proof}

Given a vector space $V$ we consider the Grassmann algebra $\bigwedge V$ as $\Z/(2)$  graded (with $V$ in degree 1) and then one has an isomorphism as $\Z/(2)$  graded  algebras
\begin{equation}\label{gas}
\bigwedge (V_1\oplus V_2)=\bigwedge V_1\widehat \otimes \bigwedge V_2\,.
\end{equation} 
In grater generality   given a supervector space $V_0\oplus V_1$ one may define the {\em super symmetric algebra of $V$}
\begin{equation}\label{sus}
\mathbb K(V):=S(V_0)\otimes \bigwedge V_1
\end{equation}where for a vector space $W$ the symbol $S(W)$ is the symmetric algebra (in practice the polynomials in the elements of a basis of $W$).\smallskip

As in Formula \eqref{gas} one has for a direct sum $U\oplus V$ of graded vector spaces:
\begin{equation}\label{gas1}
\mathbb K(U\oplus V )=\mathbb K(U)\widehat \otimes\mathbb K( V)\,.
\end{equation} 

One has the simple
\begin{theorem}\label{fc}
$\mathbb K(V)$  is a free algebra in the category  of    supercommutative superalgebras.
\end{theorem}
\begin{proof}
Exercise.
\end{proof}

This means that, given a  supercommutative superalgebra $A$ and a graded linear map $\phi:V\to A$, the map $\phi$  extends to a unique  homomorphism  of  $\mathbb K(V)$ to $A$. \medskip

{\em Supercommutative superalgebras are common in algebraic topology as they appear as cohomology algebras $H^*(X,F)$  of a topological space with coefficients in $F$, usually $F=\Z,\Q,\R,\C$ and one has an isomorphism $H^*(X\times Y,\Q)\simeq  H^*(X,\Q)\widehat\otimes H^*(Y,\Q)$.}\medskip

If $A$ is a $\Z/(2)$  graded   algebra we have that the algebra of  $n\times n$ matrices over $A$, $M_n(A)=M_n(\C)\widehat\otimes A$      is also  $\Z/(2)$  graded,  considering $M_n(\C)$  as totally in degree 0. \smallskip

 If  $A$ is supercommutative the supercenter of  $M_n(A)$ is $A=1\otimes A$.
\begin{definition}\label{grma} If  $A$ is supercommutative, an element of degree 0, resp. 1, of $M_n(A)$  is called a {\em Bosonic,} resp.  {\em Fermionic } matrix .\footnote{We use this suggestive language,  borrowed from Physics, although no Physics is in this paper.}\medskip

An element of degree 1, of $M_n(\mathbb K(V))$ is called a {\em Grassmann matrix}.
\end{definition}

\section{Superalgebras with   trace  \label{Supertraces}}  \subsection{Definitions}  In order to explain the SFT (second fundamental theorem) for Grassmann matrices we  develop a language which is a natural evolution of the one developed in the previous paragraphs, at least in characteristic 0, the language of {\em superalgebras with   trace} and their identities (cf. \cite{agpr}).\smallskip

The  SFT is a statement about the trace identities for the  Grassmann  algebra.
\begin{definition}\label{Traces1}  An associative superalgebra with   trace, also called   SAT,  over a field
$F$ is an associative $F$--superalgebra $R$ with a 1-ary operation
\[
t:R\to R
\]
which is assumed to satisfy the
following axioms:
\begin{enumerate}
\item  $t$ is $F$--linear and compatible with the $\Z/(2)$  grading.
\item  $t(a)b=(-1)^{d(a)d(b)}b\cdot t(a), \quad\forall a,b\in R $ homogeneous.
\item  $t(ab)=(-1)^{d(a)d(b)}t(ba), \quad\forall a,b\in R $  homogeneous.
\item $t(t(a)b)=t( a)t(b), \quad\forall a,b\in R.$
\end{enumerate}

\end{definition}
\noindent This operation is called a {\em formal   trace}.\smallskip

We denote  $t(R):=\{t(a),\ a\in R\}$    the image of $t$.

\begin{remark} We have the following implications:

     Axiom 1) implies that $t(R)$ is a graded $F$--submodule.

     Axiom 2) implies that $t(R)$ is contained in  the supercenter of $R$.

     Axiom 3) implies that $t$ is 0 on the space of supercommutators $[R,R]_s$.

     Axiom 4) implies that $t(R)$ is an $F$--subalgebra and that $t$ is 
$t(R)$--linear.
\end{remark}

We call  $t(R)$ the {\em   trace algebra of $R$}.  \medskip

The basic example of algebra with trace is of course the algebra of $m\times m$
matrices over a supercommutative algebra $B$ with the usual  trace.\smallskip

We have made no special requirements on the value of
$t(1)$. Of course in many important examples this is a positive integer.\bigskip

The elements   SAT form the  objects of a category, still denoted by      SAT, where   morphisms are algebra homomorphisms $\phi$ which commute with the   trace and grading:
 $$ d(\phi(a))=d(a),\quad t(\phi(a))=\phi(t(a)).$$  
Recall that in an algebra $R$ with some extra operations, an
ideal
$I$ must be stable under these operations so that $R/I$ can inherit the structure.  

In a   SAT  a superideal is an $\Z/(2)$  graded ideal
$I$    stable under the   trace, that is if $a\in I$ then $t(a)\in I$. The  usual
homomorphism theorems are  then valid for   SAT.

 \smallskip

Given a   SAT $S$, if we {\em forget the   trace} we just have an associative superalgebra, which we may denote by $F(S)$.

On the other hand, there is a simple formal construction which associates to any $F$ superalgebra $R$ a superalgebra with trace $R_t$. \smallskip

Denote by $[R,R]_s $ the linear span of supercommutators, a $\Z/(2)$ graded subspace of $R$.  Take then the quotient space $R/[R,R]_s$,  a $\Z/(2)$ graded quotient space.\smallskip

We then define $R_t:=R\widehat\otimes \mathbb K[R/[R,R]_s]$,  with  $\mathbb K[R/[R,R]_s]$ the  supercommutative superalgebra  of Formula \eqref{sus}. \smallskip

Given an element  $r\in R$, its trace $t(r)$ is the class of $r$ in $R/[R,R]_s$. \smallskip

This map extends, by linearity, to a   trace   on $R_t$ with trace algebra $\mathbb K[R/[R,R]_s]$  (notice that, in this formal construction,  $t(1)$ is just a variable).\smallskip

The    construction  $R\mapsto R_t$ is a functor, adjoint to the previous forgetful functor, that is if  $SA$  denotes the category of superalgebras, and $S\in SAT,\ R\in SA$ we have:
\begin{equation}\label{adj}
\hom_{SAT}(R_t,S)\simeq \hom_{SA}(R ,F(S))\,.
\end{equation}
\begin{proposition}\label{adj1}
A graded homomorphism $f:R\to S$  with $S$ a trace superalgebra extends uniquely to a trace preserving map $f:R_t\to S$.
\end{proposition}
\begin{proof}
This is the meaning of the previous Formula  \eqref{adj}.
\end{proof}
\subsection{Free algebras}  In algebra the notion of {\em free object} in some category is quite general. 

It applies to groups, monoids, associative or Lie algebras and  so on.  \smallskip

 For a category $\mathcal C$ of algebras of some kind  {\em a  free object over a set $X$}  is an algebra $\mathcal F$ with $X\subset \mathcal F$ and, given any algebra  $G\in\mathcal C$ and a map $f:X\to G$, this map extends to a unique algebra homomorphism  $f:\mathcal F\to G.$

For associative algebras and a set $X$, called {\em variables}, the free algebra denoted $F\langle X \rangle$  has as basis over $F$ all the {\em words}  in the {\em alphabet}\footnote{alphabet in combinatorics is any set out of which we can form {\em words} that is strings of elements of the set. }  $X$ with multiplication by justaposition.  

For $\Z/(2)$  graded algebras the construction is the same but we have 
two alphabets $$X=\{x_1,x_2,\ldots , x_r,\ldots\},\ Y=\{y_1,y_2,\ldots , y_r,\ldots\},$$ where we give degree 0 to the variables $Y$ (bosonic) and 1 to $X$ (fermionic). \smallskip

 Each monomial $M$ in $X,Y$  has a degree in $X$ and in $Y$  and as $\Z/(2)$  degree we take the degree in $X$  
 modulo 2 and denote it by $d(M)\in\{0,1\}$.\medskip

For the  {\em free superalgebra with   trace}  in the {\em variables} $X,Y$ one   applies  the   construction   in Proposition \ref{adj1}. \smallskip

\begin{lemma}\label{imm}
The $\Z/(2)$ graded space $F\langle X,Y\rangle/[ F\langle X,Y\rangle, F\langle X,Y\rangle]_s$ has as basis over $F$ the monomials up to {\em  signed cyclic equivalence}.  

The class of a monomial $M$ is denoted by $t(M)$.
\end{lemma}
\begin{proof}
Given a monomial $M=AB$  in $X,Y$ one has, by definition \ref{super}  (5) $[A,B]_s:=AB-(-1)^{d(A)d(B)}BA,$ so 
\begin{equation}\label{latra}
t(AB)=(-1)^{d(A)d(B)}t(B A)\,. 
\end{equation} This is the meaning of  signed cyclic equivalence
\end{proof}  
\medskip
 
Having denoted by $t(M)$ the signed cyclic  equivalence class of a monomial  we have that $\mathbb K[F\langle X,Y\rangle/[ F\langle X,Y\rangle, F\langle X,Y\rangle]_s]$ can be thought of as  super symmetric  algebra $\mathcal T=\mathbb K[t(M))]$, that is the algebra of polynomials in the degree 0  elements $t(M)$  tensor the exterior algebra   in the degree 1  elements $t(M)$.\smallskip

We denote by  $\mathcal F_T\langle X,Y\rangle=F\langle X,Y\rangle_t=F\langle X,Y\rangle\widehat\otimes_A \mathcal T$,\label{frea}      
the   trace of a momomial $M$ is  $t(M)$  (including $t(1)$.\smallskip

{\bf Warning}\quad Contrary to what happens in the bosonic case in this case there may be monomials $M$ with  $t(M)$=0. 

This happens when in Formula \eqref{latra}  we have $AB=BA$  and $(-1)^{d(A)d(B)}=-1.$   Notice that this may happen only for monomials of even degree.

\begin{exercise}
If $AB=BA$ there is a word $N$  such that $A,B$ are both powers of $N$.

$t(M)=0$  if and only if $M$ is an even power of a fermionic word.
\end{exercise}
\begin{proposition}\label{Fra}
By Formula  \eqref{adj} and Proposition \ref{adj1}  we have  that the algebra $\mathcal F_T\langle X,Y\rangle=F\langle X,Y\rangle_t=F\langle X,Y\rangle\widehat\otimes_A \mathcal T$ is  the free superalgebra with trace in the variables  $X,Y$.  
\end{proposition}
 
\subsection{Strings and encoding map}\qquad {\em This paragraph establishes the notations needed to discuss the invariant theory.}\medskip

  We now discuss the {\em encoding map}  which describes the subspace  $\mathcal F_T(e,f)$ of the free algebra $\mathcal F_T\langle X,Y\rangle$,   of the elements         which are multilinear in the variables $x_1,\cdots,x_e;y_1,\cdots,y_f$,  in terms of the symmetric  group $S_{e+f+1}$.
  
 Respectively the subspace  $ \mathcal T(e,f)$ of  its trace algebra $\mathcal T$,  which are multilinear in $x_1,\cdots,x_e;y_1,\cdots,y_f$  in terms of the symmetric  group $S_{e+f}$.  \bigskip

Consider a set $A:=\{a_1<a_2<\cdots< a_m\}\subset \N^+$ of positive integers with $|A|=m$.
\begin{definition}\label{colo}\strut
\begin{enumerate}\item A coloring of $A$ is a map  $C:A\to\{0,1\}$.  The elements with $C(a)=0$ are called {\em bosonic} while those with $C(a)=1$ {\em fermionic}.\smallskip

\item A string with support in $A$  is a bijection  from $\{1,2,\cdots,m\}$ to $A$ and it is displayed as a word  $w=a_{\sigma(1)},a_{\sigma(2)},\cdots, a_{\sigma(m)}$.  This word can also be interpreted as a permutation $\sigma$ of $A$.\smallskip

\item For such a word $w$ and coloring $C$ set  $ w^f_C $    to be the substring  formed by the fermionic elements.\smallskip

\item 
We set $d_f^C(w):=|w^f_C|$ its cardinality.  \smallskip

\item 
The  sign $\epsilon_C(w)$ of  $w$ is the sign of the permutation which reorders the fermionic string $w^f_C$.
\end{enumerate}

\end{definition}

A permutation $\sigma$ of $A$   can also be displayed in   its  cycle form. 

This second display is not unique since   $(A,B)=(B,A)$.\smallskip

The canonical way of displaying a permutation $\sigma$ of $A$  is to start from the first element of $A$  write the string of its cycle, close it with a parenthesis, open a parenthesis write the first of the elements of $A$  which have not appeared and iterate the procedure.

In this display the first element of each cycle is minimal in the cycle and the first elements of the  various cycles are increasing from left to right.\smallskip

We  denote the corresponding canonical parenthesized  string or p--string as $\lambda(\sigma)$.\smallskip

We also need a second string. \label{lamb}

 Start from $\lambda(\sigma)$,  single out the cycle which contains the last element $a_m$ of $A$  displayed as $(Ba_m)$,  then remove this cycle  and place the string $B$ at the end. Denote this $\mu(\sigma)$.
 
 \begin{remark}\label{mus}
Observe that, if  $a_m$ is a fixed point of  $\sigma$,  then  $\mu(\sigma)=\lambda(\sigma')$ where $\sigma'$ is the restriction of the permutation $\sigma$ to $A\setminus \{a_m\}$.
\end{remark}
\begin{example}\label{lam} $A=\{1,2,\cdots,8,9\};\  \sigma= 4,1,8,2,3,7,6,9,5$
$$\lambda(\sigma)=(1,4,2)(3,8,9,5)(6,7)\implies \mu(\sigma)=(1,4,2)(6,7)     5,3,8 \,.$$

$$\lambda(\tau)=(1,4,2)(3, 9,5)(6,8,7)\implies \mu(\tau)=(1,4,2)(6,8,7)5,3  \,.$$
\end{example}
\bigskip

\subsubsection{Colors}\quad Let $\mathcal F_T$  be as before the free superalgebra with   trace in infinitely many fermionic variables $x_i$ and  bosonic variables $y_i$.

Given a subset $A\subset \N$ with a coloring $C$ and a string  $w=a_1a_2\cdots a_m$  we evaluate this string in $\mathcal F_T$  by setting 
\begin{equation}\label{tau}
 \tau_C(j)=\begin{cases}
y_j\quad\text{if}\ j\ \text{is bosonic}\\x_j\quad\text{if}\ j\ \text{is fermionic}\\  \tau_C(w)=\epsilon_C (w)  \tau_C(a_1)  \tau_C(a_2)\cdots   \tau_C(a_m)
\end{cases}
\end{equation}
For a p--string $w=(w_1)(w_2)\cdots w(r)$  we set, use Definition \ref{colo}  (5):
\begin{equation}\label{taua}
\tau_C(w)=\epsilon_C (w) tr( \tau_C(w_1)) tr( \tau_C(w_2))\cdots  tr( \tau_C(w_r))\in \mathcal T=tr(\mathcal F_T).
\end{equation}  We also consider the partial p--strings $w=(w_1)(w_2)\cdots w(r)w_{r+1}$ and their colored evaluation 
  \begin{equation}\label{taue}
\tau_C(w)=\epsilon_C (w) tr( \tau_C(w_1)) tr( \tau_C(w_2))\cdots  tr( \tau_C(w_r)) \tau_C(w_{r+1})\in  \mathcal F_T
 \,.
\end{equation}  \medskip
 Observe that, if  $w_{r+1}=1$  we obtain the Formula \eqref{taua}.\smallskip

 To a p--strings $w=(w_1)(w_2)\cdots w(r)(w_{r+1}a_m)$ we associate  the partial p--strings $\bar w=(w_1)(w_2)\cdots w(r)w_{r+1}$.\medskip
 
 Now  let $A=\{1,2,\cdots,e,e+1,\cdots,e+f\}$  and color  $$C_{e,f}(i)=\begin{cases}
0\quad \text{if}\quad i\leq e\\
1\quad \text{if}\quad i> e
\end{cases}$$
It is then convenient to redefine   $ \tau_{C_{e,f}}$  on the fermionic indices $e+j$  by setting  $ \tau_{C_{e,f}}(e+j)  :=x_j$.

In this way p--strings in $A$  correspond to trace monomials in the variables $\{y_1,y_2,\cdots, y_e,x_1,\cdots,x_f\}$.\medskip
             
\begin{remark}\label{alte}
There is an alternative approach to the previous construction which we will use later, For each  $a\in A$  with $C(a)=1$  add a Grassmann variable  $\epsilon_i$ which we assume to supercommute with all the $x_i$.  Then the elements $\epsilon_ax_a$  should be considered as bosonic  and then by setting 
\begin{equation}\label{alte1}
 \tau_C^\epsilon(j)=\begin{cases}
y_j\quad\text{if}\ j\ \text{is bosonic}\\ \epsilon_jx_j\quad\text{if}\ j\ \text{is fermionic}\\  \tau_C^\epsilon(w)=  \tau_C^\epsilon(a_1)  \tau_C^\epsilon(a_2)\cdots   \tau_C^\epsilon(a_m)
\end{cases}
\end{equation}
If $j_1<j_2<\cdot<j_r$ are the Fermionic indices appearing in $w$  we have 
\begin{equation}\label{tau1}
\epsilon_{j_1}\epsilon_{j_2}\cdots \epsilon_{j_r}  \tau_C(w)= \tau_C^\epsilon(w)\end{equation}
For a p--string $w=(w_1)(w_2)\cdots w(r)$  we have:
\begin{equation}\label{taua}
\epsilon_{j_1}\epsilon_{j_2}\cdots \epsilon_{j_r} \tau_C(w)=  tr( \tau_C^\epsilon(w_1)) tr( \tau_C^\epsilon(w_2))\cdots  tr( \tau_C^\epsilon(w_r))$$$$\in \mathcal T[\epsilon_{j_1},\epsilon_{j_2},\cdots ,\epsilon_{j_r}]=tr(\mathcal F_T)[\epsilon_{j_1},\epsilon_{j_2},\cdots ,\epsilon_{j_r}].
\end{equation} 
\end{remark} \smallskip

\begin{definition}\label{mult}
We define $\mathcal T(e,f)\subset \mathcal T$        as the space of elements of   $ \mathcal T$   which  which are multilinear  in the variables $y_1,\cdots,y_e;x_{ 1},\cdots,x_{ f}$.\medskip

We define $\mathcal F_T(e,f)\subset \mathcal F_T$        as the space of elements of   $\mathcal F_T$   which   are multilinear  in the variables $y_1,\cdots,y_e;x_{ 1},\cdots,x_{ f}$.
\end{definition}
 \begin{lemma}\label{labi}
The map  $G\mapsto tr( G y_{e+1})$ is a linear isomorphism between  $\mathcal F_T(e,f)$ and $\mathcal T(e+1,f)$.\smallskip

The map  $G\mapsto tr(G x_{f+1})$ is a linear isomorphism between  $\mathcal F_T(e,f)$ and $\mathcal T(e,f+1)$.
\end{lemma}
\begin{proof}
Exercise.
\end{proof}
\begin{proposition}\label{ilmpsi} [Encoding map]\strut 

Take a permutation $\sigma\in S_{e+f}$  written  as a p--string, \smallskip

$ \tau_{C_{e,f}}(\sigma)$  depends only on $\sigma\in S_{m}$ and not on its display as  p--string. \smallskip

 We  set $\Phi_\sigma^{e,f}:= \tau_{C_{e,f}}(\sigma)$.\medskip

For a permutation $\delta\in S_{e+f+1}$  written  as a parenthesized string  $$(w_1)(w_2)\cdots w(r)(w_{r+1},m+1)$$ in which   the last  cycle is a string ending with $m+1$, we set  $\bar\delta$   the  partial p--string  $w=(w_1)(w_2)\cdots w(r)w_{r+1}.$     \medskip

 $ \tau_{C_{e,f}}(\bar\delta)$  depends only on $\delta\in S_{e+f+1}$ and not on its display as   p--string.  \smallskip

We  set $\Psi_\delta^{e,f}:= \tau_{C_{e,f}}(\delta)$.\medskip

\begin{enumerate}
\item The map $\sigma\mapsto \Phi ^{e,f}_\sigma=\tau_{C_{e,f}}(\sigma)$ is a linear isomorphism  between the group algebra $\C[S_{e+f}]$ with the subspace     $\mathcal T(e,f)$ of  its trace algebra $\mathcal T$                which are multilinear  in the variables $y_1,\cdots,y_e;x_{ 1},\cdots,x_{ f}$.\medskip


\item The map $\sigma\mapsto \Psi^{e,f}_\sigma= \tau_{C_{e,f}}(\bar\sigma)$ is a linear isomorphism  between the group algebra $\C[S_{e+f+1}]$  with the space $\mathcal F_T(e,f)$ of the free algebra $ \mathcal F_T\langle X,Y\rangle$            which are multilinear in the variables $y_1,\cdots,y_e;x_{ 1},\cdots,x_{ f}$.
\item  $\sigma\in S_{e+f+1}$ then $ \Phi^{e,f}_\sigma=tr(\Psi^{e,f}_\sigma x_{f+1})=tr(\Psi^{e-1,f+1}_\sigma y_{e})$.  \end{enumerate} 
.\medskip

 \end{proposition}
\begin{proof}
The fact that the two maps  depend only on $\sigma$ is a consequence of the axioms on the trace in $\mathcal F_T$. The fact that they are linear isomorphisms depends on the definition of multilinear.
\end{proof}  
\begin{definition}\label{enm}
$ \tau_{C_{e,f}}$ is called the {\em encoding map}.
\end{definition}
Unless there is a risk of confusion  we drop the pedex $^{e,f}$ from the formulas.
 \begin{example}\label{unes}
$m=3+4=7,\ \sigma =  3,1,5,7,2,6,4=(1,3,5,2)(4,7)(6)$
$$\Phi_\sigma= tr(y_1y_3x_2y_2)tr( x_1x_4)tr(x_3) $$
$$\Psi_\sigma= tr(y_1y_3x_2y_2) tr(x_3)x_1 $$
$m=4+3=7,\ \sigma =  3,1,5,7,2,6,4=(1,3,5,2)(4,7)(6)$
$$\Phi_\sigma=- tr(y_1y_3x_1y_2)tr( y_4x_3)tr(x_2) $$
$$\Psi_\sigma= tr(y_1y_3x_1y_2) tr(x_2)y_4 $$
\end{example}
\begin{remark}\label{att}
Beware  of the  fact that for an even cycle  $\mathtt c$  the  formula for $ \phi_{\mathtt c }$ may depend  also  from the way in which $\mathtt c $ is displayed, since $ tr(\eta_{ 1}\circ \eta_{ 2})=\pm tr(\eta_{2}\circ \eta_{1})   $  while $(1,2)=(2,1)$.

For a permutation  it also depends on the order in which we write the cycles.
\end{remark}
One can use  the convention $\lambda(\sigma)$  of page \pageref{lamb}.
\begin{example}\label{phis} $m=0+5=5$, 
$\sigma =(2,4,1)( 5,3) =(1,2,4 )(3,5)$ then
$$\Phi_\sigma= tr(x_2x_4x_1)tr(x_5x_3)=-tr(x_1x_2x_4 )tr(x_3x_5) $$
$$\Psi_\sigma= tr(x_2x_4x_1)tr(x_5x_3)=-tr(x_1x_2x_4 )  x_3  $$
\end{example}  
the sign is that of the permutations $  2,4,1,5,3 ,\ 1,2,4,3,5 $. \medskip

\begin{remark}\label{symm}\strut \begin{enumerate}\item In the free algebra a particular  form of substitution is by permuting the bosonic  or the fermionic variables.\smallskip

So there is an action of the  group $S_e\times S_f$ which can also be identified to a subgroup of $S_m$.

\item The encoding map has the following symmetry property. \smallskip

\item  For $\tau=(\alpha,\beta)\in S_e\times S_f$, the permutation $\alpha$ acts on bosonic indices and   $\beta$   acts   on the  fermionic indices  $1,2,\ldots,f$:
\begin{equation}\label{simme}
\Phi_{\tau\circ\sigma\circ\tau^{-1}}(y_1,\cdots,y_e;x_{ 1},\cdots,x_{ f})=\epsilon_\beta(y_{\alpha(1)},\cdots,y_{\alpha(e)};x_{\beta( 1)},\cdots,x_{\beta(f)})\,.
\end{equation}\end{enumerate}

\end{remark}

\subsection{Trace identites}\qquad

By definition of free superalgebra given any $F$--superalgebra with trace  $R=R_0\oplus R_1$  the  homomorphisms  from $\mathcal F_T= F\langle X,Y\rangle[t(M)]$ to $R$  (trace preserving) are given by {\em evaluations} of polynomials associated to maps $X\to R_1,\ Y\to  R_0$.

Using evaluations allows us to introduce  the notion of {\em trace identity} for a trace superalgebra $R$.

\begin{definition}\label{trid}
A  trace identity  for a trace superalgebra $R$   is a element $\mathtt f\in \mathcal F_T\langle X,Y\rangle=F\langle X,Y\rangle[t(M)]$ of the free algebra such that $\mathtt f$    vanishes under all evaluations in $R$.

\end{definition}  
  
   Let us now remark that 
the set of all trace identities of a trace algebra $R$ is a special ideal of $\mathcal F_T$, called a $T$--ideal.
\begin{definition}\label{tid}
A  $T$--ideal of the free algebra with trace $\mathcal F_T$  is a super ideal of  $\mathcal F_T$  closed under  all substitutions of variables.
 
By definition  a substitution of  a fermionic variable $x$  must be an element of degree 1 in $\mathcal F_T$, while a substitution of  a  bosonic variable $y$  must be an element of degree 0 in $\mathcal F_T$.\end{definition}

A basic question for any algebra $R$ is to describe a set of generators  of its trace identities as $T$--ideal, this is in general a very difficult question and it is a remarkable fact that for matrices one can obtain a simple answer.\medskip

The theory is really effective only when $F$ is a field of characteristic 0, so from now on, in this section,  we assume that this is the case  (the reader can assume $F=\C$ or $\Q$).\medskip

\subsubsection{Varieties}

A subset $S\subset  F\langle X,Y\rangle\widehat\otimes_A \mathcal T$ generates a $T$--ideal $I_S$  and the {\em variety} $\mathcal V_S$ of all trace superalgebras satisfying all the  trace identities $S$.

Of course the superalgebras in  $\mathcal V_S$ satisfy also all the trace identities in the $T$--ideal  $I_S$.

The quotient $ F_S:=F\langle X,Y\rangle\widehat\otimes_A \mathcal T/I_S$ satisfies $I_S$  and any identity of $ F_S$ is also in $I_S$.\smallskip

One easily sees that $F_S$  is a free algebra, in the classes of the elements  $x,y$  for the variety $\mathcal V_S$.

The classes of  $x,y$ are also called {\em generic elements} for   the variety $\mathcal V_S$.\medskip

\begin{remark}\label{taut}\quad

A tautological statement is thus  that, an element $\mathtt f\in F\langle X,Y\rangle\widehat\otimes_A \mathcal T$  is a trace identity for  the variety $\mathcal V_S$ if and only if it vanishes when evaluated in the generic elements.
\end{remark}
Let $I_S\neq \{0\}$  be a $T$--ideal.
Given a trace superalgebra $R\in \mathcal V_S$, consider the trace superalgebra $M_n(R)$.  By a Theorem of Regev the $T$--ideal  $I_{n,R}$ of   trace identities of $M_n(R)$ is also $\neq \{0\}$.   

Let $F_{S,n}$  denote the free superalgebra  for the variety associated to $I_{n,R}$, then  the corresponding free algebra 
$ F_{n,R}\langle X,Y\rangle:=F\langle X,Y\rangle\widehat\otimes_A \mathcal T/I_{n,R}$ can be described by the method of {\em generic matrices}.\footnote{Although I introduced generic matrices in my 1966 Ph thesis I am not aware if this remark has been made before in this generality.}
\medskip

Consider  the free algebra  $F \langle X(n),Y(n)\rangle\otimes_A \mathcal T$ in the  
variables
\begin{equation}\label{matva}
X(n):=\{x_{h,k}^{(i)} \},\quad Y(n):=\{y_{h,k}^{(i)}  \};\quad  h,k=1,\cdots,n,\ i=1,2,\cdots
\end{equation}

Define the {\em generic matrices}
\begin{equation}\label{genma}
X_i:=(x_{h,k}^{(i)}),\quad Y_j:=(y_{h,k}^{(j)}),\ i,j=1,2,\cdots
\end{equation}

We then have an injective map
$$\mathtt i: F\langle X,Y\rangle\otimes_F \mathcal T\mapsto M_n(  F\langle X(n),Y(n)\rangle\otimes_F \mathcal T),\  \mathtt i:\ x_i\mapsto X_i,\ \mathtt i:\ y_j\mapsto Y_j\,.$$

Given any evaluation  $\pi:F\langle X,Y\rangle\widehat\otimes_A \mathcal T \to M_n(R)$, let $$\pi(x_i)=(a_{h,k}^{(i)}), a_{h,k}^{(i)}\in R_1;\ \quad\pi(y_j)=(b_{h,k}^{(j)}),\ b_{h,k}^{(j)}\in R_0\,.$$ 

Consider the induced evaluation  $\bar\pi:F\langle X(n),Y(n)\rangle\otimes_A \mathcal T\to R$  where $\bar\pi:\ x_{h,k}^{(i)}\mapsto a_{h,k}^{(i)},\quad \bar\pi:\ y_{h,k}^{(j)}\mapsto b_{h,k}^{(j)}$. 
We have the factorization
 \begin{equation}\label{fatt}
\xymatrix{  F\langle X,Y\rangle\widehat\otimes_A \mathcal T\ar@{->}[r]^{\mathtt i\qquad}\ar@{->}[rd]_\pi&M_n(  F\langle X(n),Y(n)\rangle\otimes_F \mathcal T)\ar@{->}[d]^{M_n(\bar \pi)}\\
&M_n(R) } .
\end{equation} 
and clearly the evaluation $\pi$ vanishes    if and only if $\bar\pi$ vanishes.  \medskip

Both maps factor through the free algebras  in the varieties of $M_n(R)$ and of $R$ respectively giving the commutative diagram:
\begin{equation}\label{fatt1}
\xymatrix{  F\langle X,Y\rangle\otimes_F \mathcal T\ar@{->}[r]^{\mathtt i\qquad}\ar@{->}[d]_\rho&M_n(  F\langle X(n),Y(n)\rangle\otimes_F \mathcal T)\ar@{->}[d]^{M_n(\bar \rho)}\\
  F\langle X,Y\rangle\otimes_F \mathcal T/I_{M_n(R)}\ar@{->}[r]^{\bar{\mathtt i}\qquad}\ar@{->}[rd]_{\pi'}&M_n(  F\langle X(n),Y(n)\rangle\otimes_F \mathcal T /I_R)\ar@{->}[d]^{M_n(\bar \pi')}\\
&M_n(R) } .
\end{equation}  With $\rho,\ \bar\rho$ the respective quotient maps and $\pi=\pi'\circ\rho,\ \bar\pi=\bar\pi'\circ\bar\rho$.  One can apply this construction to the superalgebra $R= F\langle X(n),Y(n)\rangle\otimes_F \mathcal T /I_R$ and $\pi=\mathtt i\circ\rho,\ \pi'=\mathtt i$.

We obtain that $\bar{\mathtt i}$ is injective so it is an isomorphism with the trace subalgebra  of  $M_n(  F\langle X(n),Y(n)\rangle\otimes_F \mathcal T /I_R)$ generated by the images $\bar{\mathtt i}  $ of the elements $\rho(x_i),\ \rho(y_i)$, which are the images under  $M_n(\bar \rho)$ of the {\em matrix variables } $X_i,\ Y_j$.  \medskip

We have proved
\begin{theorem}\label{gamm} 
The trace subalgebra  of  $M_n(  F\langle X(n),Y(n)\rangle\otimes_F \mathcal T /I_R)$ generated by the  images (generic matrices) under  $M_n(\bar \rho)$ of the {\em matrix variables } $X_i,\ Y_j$ is a free algebra in the variety generated by $M_n(R)$.
\end{theorem}

\section{Matrices}\quad
\subsection{Generic matrices}

We want to apply Theorem \ref{gamm} to $R=\mathbb K(V)$ which a free supercommutative superalgebra, and show that the free algebra in the variety generated by  $M_n(R)$  is described as an algebra of invariants, Theorem \ref{FFT}.\medskip

Let $G=GL(n,\C)$   be the linear group  acting by conjugation on matrices.
This action factors through the projective linear group $\mathtt P_n=GL(n,\C)/\C^*$.\smallskip 

This action induces an action of $G$   on the vector space  $V=(M_n )^{ \oplus m}$ and on its dual  $V^*=(M_n^*)^{ \oplus m}$  of linear functions on $m$--tuples of matrices.\medskip

Decompose $m=e+f$ and  give  a structure  to $V$, and so also to $V^*$, of graded vector space, by considering the first $h$ summands  in degree 0, the {\em bosonic matrices}  and the  remaining summands of degree 1, the {\em fermionic matrices}.  \smallskip

We then have that the action of $G$ extends as automorphisms to  $\mathbb K(V^*)$.  
 We also take  $\mathbb K(V^*)\otimes M_n(\C)=M_n(\mathbb K(V^*))$, with the diagonal action of $G$ still an    action   as automorphisms.
 \smallskip

We want to describe the invariant subalgebras $\mathbb K(V^*)^G, M_n(\mathbb K(V^*))^G$ of invariants for this action.  

Observe that $M_n(\mathbb K(V^*))^G$ is a SAT  with $\mathbb K(V^*)^G$ as trace algebra.\medskip

The vector space $ V_0$  has as a basis  $ \{\eta_{h,k}^{(i)},\ h,k=1,\cdots,n, \ i=1,2,\cdots \}$, the  $\eta_{h,k}^{(i)}$ are bosonic variables.

 The vector space $ V_1$  has as a basis    $ \{\xi_{h,k}^{(i)},\ h,k=1,\cdots,n, \ i=1,2,\cdots \}$ the   Grassmann  (fermionic) variables.\smallskip

According to Theorem \ref{gamm} they give rise to  $e+f$  generic Grassmann (bosonic and fermionic) matrices \begin{align}\label{grava}
& \eta_1,\cdots, \eta_e;\  \eta_i=(\eta_{h,k}^{(i)})=\sum_{h,h}e_{h,k}\otimes \eta_{h,k}^{(i)}\in M_n(\C)\otimes S(V_0)\\& \xi_1,\cdots, \xi_f;\  \xi_i=(\xi_{h,k}^{(i)})=\sum_{h,k}e_{h,k}\otimes \xi_{h,k}^{(i)}\in M_n(\C)\otimes \bigwedge V_1
\end{align} of size $n$.

\begin{remark}\label{action}
The action of $G$  is defined so that the matrices  $$g\cdot  \eta_i=\sum_{h,h}g\cdot e_{h,k}\otimes g\cdot \eta_{h,k}^{(i)},\ g\cdot  \xi_i=\sum_{h,h}g\cdot e_{h,k}\otimes g\cdot \xi_{h,k}^{(i)}$$  are $G$--invariant. They are   {\em generic matrices} as in Theorem \ref{gamm}   or in page \ref{genn} for 1 variable.
\end{remark}

   We set  $S=S_{e,f,n}=\mathbb K(V^*)=S[\eta_{h,k}^{(i)}]\otimes \bigwedge[\xi_{h,k}^{(i)}] $.\medskip
   
   We   want to study  the $G$  invariants  of  the   algebra $S_{e,f,n}$  and  of  $n\times n$ matrices  $M_n(S_{e,f,n})=M_n(\C)\otimes  S_{e,f,n} $ under the diagonal tensor product action.\medskip
   
   For this consider the evaluation of the free superalgebra with trace:
   
    $\mathcal F_T\langle X,Y \rangle=F\langle X,Y \rangle\otimes\mathcal T$, with $X=x_1,x_2,\ldots , x_f ; \ Y=y_1,y_2,\ldots , y_e$ in  
$M_n(S_{e,f,n})=M_n(\C)\otimes  S_{e,f,n} $  given by $ev:y_i\mapsto \eta_i$ and $ev:x_i\mapsto \xi_i$. 

We want to prove:
 \begin{theorem}\label{FFT} [FFT]\quad The previous evaluation map  surjects to the algebra of $G$--invariants. That is:\medskip
 
 The algebra  $\mathcal T_n$ of invariants  of  both bosonic and fermionic matrices is  generated by the elements $tr(M)$ for all monomials $M$ in  the generic matrices $\xi_i,\eta_j$.\medskip
 
 The algebra  $\mathcal E_n$ of equivariant  maps  of  both bosonic and fermionic matrices is generated by the elements $tr(M)\in \mathcal T_n$ and by the generic matrices $\xi_i,\eta_j$.
 
 \end{theorem}
 \begin{proof}

 We start by observing that $ev$ maps  to the invariants.
 \begin{lemma}\label{bain}
The elements $ \eta_i, \xi_j$  are invariant. 

The map    trace $tr: M_n(\C)\otimes  S_{e,f,n}\to   S_{e,f,n},\ tr(A\otimes s)=tr(A)s$  is equivariant.
\end{lemma} 
\subsubsection{Polarization}\qquad  Using the classical method of polarization of Aronhold we can reduce to prove  that the map $ev$ is surjective 
  for multilinear  invariants.

Classically one defines the {\em polarizations}  of a  polynomial  $f(u_1, u_2,\ldots,u_k) $ where the $u_i$ are vector variables and $f$ is homogeneous of degree $h_i$ in $u_i$.

Let us start with   a single variable, say $u_1$ and for simplicity assume $f$ is homogeneous of degree $h$ in $u_1$. 
%
Substitute   $u_1$ with $\sum_{i=1}^h\lambda_iv_i$ where $\lambda_i$ are parameters, that is formal commutative indeterminates  added to the base field $F$ and $v_i$ new vector variables.
Then
$$f(\sum_{i=1}^h\lambda_iy_i, u_2,\ldots,u_n)=\!\!\!\!\!\sum_{a_1+a_2+\ldots+a_h=h}\!\!\!\!\lambda_1^{a_1}\ldots \lambda_h^{a_h }f_{a_1,\ldots,a_h}(v_1,\ldots,v_h, u_2,\ldots,u_n)$$
  \smallskip

In particular  the polynomial with all $a_i=1$ is multilinear and symmetric in the $v_i$ and it is called the {\em full polarization}  of   $f$ with respect to the variable $u_1$.\medskip

Of course one can polarize all the variables obtaining a multilinear polynomial symmetric in the various groups of auxiliary variables.
 
In general polarization works in any free algebra  where  we can make the substitution  for $u_1$ with $\sum_i\lambda_iv_i$.
 
We leave to the reader to verify
\begin{exercise}\label{resf0} Let  $f(u_1, u_2,\ldots,u_k)$ be homogeneous of degree $h_i$ in $u_i$.

 If in the full polarization of $f(u_1, u_2,\ldots,u_n)$  with respect to the variable $u_1$  we set all the new variables $v_i$ equal to $u_1$ we obtain $h! f(u_1, u_2,\ldots,u_n) $.

\end{exercise}
This last operation is called {\em restitution}. \medskip

The main remark is that if $f$ is invariant  under some group $G$ then all  polarizations are also invariants, so if we know the multilinear invariants we have formulas for all invariants by restitution.\smallskip

There is a subtle point  in our case having to do with the encoding map.  Before taking the issue in general let me give an example.
$$Tr(xy^2)\mapsto Tr(x(y_1y_2+y_2y_1) )=\Phi^{2,1}_\sigma+ \Phi^{2,1}_{\sigma^{-1}},\ \sigma= (1,2,3)\,.$$
  $$Tr(yx^2)\mapsto Tr(y(x_1x_2+x_2x_1) )=\Phi^{1,2}_\sigma-\Phi^{1,2}_{\sigma^{-1}},\ \sigma= (1,2,3)\,.$$
What we see is that the two formulas appear symbolically symmetric in the polarized variables, but decoding them  one has an element of the group algebra which in the bosonic case is symmetric under conjugation while in the fermionic case antisymmetric.

This is a general fact and depends upon Formula  \eqref{simme}.

\begin{proposition}\label{siman}
Take a formal trace polynomial homogeneous $\mathtt f$ in the variables $x_i$ of degree $h_i$  and  in the variables $y_j$ of degree $k_j$.

Fully polarize $\mathtt f$ in all the variables obtaining an expression which is symmetric in each of the sets of polarized variables.

Then decode this expression as an element of the group algebra.  Then the various groups of bosonic indices are conjugation invariant while 
 the various groups of fermionic  indices are conjugation  antisymmetric.\end{proposition}
This proposition gives a description of  multihomogeneous elements in terms of multilinear ones.

In particular one may interpret the fact that, if a formal expression has degree $>n^2$  in a single fermionic variable, then it vanishes on $n\times n$ matrices.
\subsubsection{ Multilinear  invariants of matrices.}   To compute multilinear  invariants of matrices  we use a variation of the classical {\em symbolic method}  suggested by the Formulas of Kostant \cite{kostant}, Lemma 4.9.\smallskip
  
 Multilinear  elements can also be  defined formally as follows: 
 
\noindent we have  the diagonal $m\times m$ matrices  $d=(d_1,\cdots,d_m)$ acting on $\xi_{h,k}^{(i)}$  by   $d(A_1,\cdots,A_m)=(d_1A_1,\cdots,d_mA_m)$ this action commutes with that of $G$.\medskip
  
  The multilinear elements of  $ M_n(\C)\otimes  S_{e,f,n},\    S_{e,f,n}$ are the ones for which $d\cdot a=\prod_id_ia$.\smallskip

We have   $$S= \mathbb K(V^*)= \underbrace{S(M_n^*)\ \otimes \cdots \otimes S(M_n^* ) }_{e \; \rm{times}}\otimes \underbrace{\bigwedge M_n^*\widehat \otimes \cdots\widehat\otimes \bigwedge M_n^*  }_{f \; \rm{times}}$$ and the multilinear elements  are   $$ \mathbb K(V^*)_{mult}=\left((M_n^*)^{ \otimes m}\right)=\underbrace{  M_n^* \otimes \cdots \otimes  M_n^*  }_{m \; \rm{times}},$$ we thus  want to understand a symbolic formula  for   $\left((M_n^*)^{ \otimes m}\right)^G\subset  \mathbb K(V^*)$.  \smallskip

Similarly  for  the multilinear elements of  $ M_n(  \mathbb K(V^*)) $  which are identified to $M_n(\C)\otimes \underbrace{  M_n^* \otimes \cdots \otimes  M_n^*  }_{m \; \rm{times}}$.\smallskip

Using Lemma \ref{labi}
 and Proposition  \ref{ilmpsi} we will reduce the description of  equivariant maps to that of invariants.\medskip

We have  the classical Schur--Weyl Theory (see Theorem 6.1.1 of \cite{agpr}):
\begin{theorem}\label{shw}
$$(M_n^*)^{ \otimes m}\simeq M_n^{ \otimes m}=End(\C^{\otimes m}),\implies  \left((M_n^*)^{\otimes m}\right)^G\simeq End_G(\C^{\otimes m})\,.$$   The commuting algebra   $End_G(\C^{\otimes m})$ is linearly generated by the symmetric group  $S_m$ acting on  $\C ^{\otimes m}$  by
\begin{equation}\label{siac}
\sigma (v_1\otimes v_2\otimes \cdots\otimes v_m)=v_{\sigma^{-1} (1)}\otimes v_{\sigma^{-1} (2)}\otimes \cdots\otimes v_{\sigma^{-1} (m)}\,.
\end{equation}
\end{theorem}

\begin{remark}\label{duap}
The isomorphism $(M_n^*)^{ \otimes m}\simeq M_n^{ \otimes m}$ is induced by the non degenerate duality pairing  between $End(\C^{\otimes m})$ and its dual $End(\C^{\otimes m})^*$  given by  $tr(AB)$.
\end{remark}

Since the invariants  $\left(M_n^{ \otimes m}\right)^G$  are encoded  by  elements of the symmetric group  so we have to understand a permutation $\sigma$  as element of  the multilinear elements $\left((M_n^*)^{\otimes m}\right)^G\subset  \mathbb K(V^*)^G$.


\begin{lemma}\label{iden}

Consider   the map 
$\phi:M_n(\C)^{\otimes m}\to   \mathbb K(V^*)  $ defined by
\begin{align}\label{imm}
&\phi:A_1\otimes\cdots \otimes A_m\mapsto tr(A_1\eta_1\otimes\cdots \otimes A_e\eta_e\otimes A_{e+1} \xi_1\otimes\cdots \otimes A_m \xi_f)\\&=tr(A_1\eta_1)\cdots tr(A_e\eta_e)tr(A_{e+1}\xi_1) \cdots tr(A_{m}\xi_f).
\end{align} this map is injective, $G$ equivariant and its image is the space of elements of $ \mathbb K(V^*) $  which are multilinear in the $\eta_i, \xi_j$,  \end{lemma} 
\begin{proof} We are using the duality pairing of Remark \ref{duap}. Given  $A=(a_{i,j})$  and $\eta=(\eta_{i,j})$ bosonic, $\xi=(\xi_{i,j})$  fermionic matrix  we have
$$ tr(e_{h,k}\cdot \eta)= \eta_{k,h} ,\ tr(e_{h,k}\cdot \xi)= \xi_{k,h} \,.$$  

\end{proof}  
\begin{proposition}\label{sima}
The map $\phi$  identifies the invariants  in  $M_n(\C)^{\otimes m}$, which are the linear span of $S_m$,  to the multilinear invariants in $ \mathbb K(V^*)  $.
\end{proposition}

\medskip

For multilinear invariants in $ M_n(\mathbb K(V^*))  $ we use the usual trick, we use Lemma \ref{labi}, we add another generic matrix, which  can be either bosonic of fermionic, let us say it is  $\eta_{e+1}$    then we have a bijective map between multilinear $G$  invariants in $\eta_1,\cdots, \eta_e; \xi_1,\cdots,\xi_f$  of $ M_n(\mathbb K(V^*))  $ and multilinear $G$  invariants in $ \eta_1,\cdots,  \eta_e, \eta_{e+1}; \xi_1,\cdots,\xi_f$  of $ M_n(\mathbb K(V^*))  $ by the formula
\begin{equation}\label{corr}
f( \eta_1,\cdots,  \eta_e;  \xi_1,\cdots, \xi_f)\mapsto tr( f( \eta_1,\cdots,  \eta_e;  \xi_1,\cdots, \xi_f) \eta_{e+1})\,.
\end{equation} We will see how, knowing a Formula for $tr( f( \eta_1,\cdots,  \eta_e;  \xi_1,\cdots, \xi_f) \eta_{e+1})$ we can decode it   as   a Formula for $  f( \eta_1,\cdots,  \eta_e;  \xi_1,\cdots, \xi_f) $.\medskip

\begin{definition}\label{phis}
\begin{equation}\label{sg}
\phi_\sigma^{e,f}:=  tr(  \eta_1\otimes\cdots \otimes  \eta_e\otimes   \xi_1\otimes\cdots \otimes  \xi_f \circ \sigma^{-1}),\ \sigma\in S_m.
\end{equation}
\begin{equation}\label{sg1}
\psi^{e,f}_\sigma :\quad\phi^{e,f}_\sigma=  tr( \psi_\sigma \eta_{e+1}),\ \text{or}\ tr( \psi_\sigma \xi_{f+1}).\ \sigma\in S_{m+1}.
\end{equation}
\end{definition}
\begin{proposition}\label{lphi}
The multilinear  $G$ invariants in   $ \mathbb K(V^*)$, by Lemma \ref{iden},  are linearly generated by the elements $\phi_\sigma$.
\smallskip

The multilinear  $G$ invariants in   $M_n( \mathbb K(V^*))$, by Lemma \ref{iden},  are linearly generated by the elements $\psi_\sigma$.

\end{proposition}
\subsection{Normalization}
If $m>n$     an expression of a multilinear invariant as linear combination of the $\phi_\sigma$   is not unique.\smallskip

In fact the algebra homomorphism $\C[ S_m]\to End_G(\C^{\otimes m})$,  as soon as $m>n$, has a non trivial Kernel $K_{m,n}$.

So, in order to make the description of the multilinear invariants unique  we have  to choose a basis  
   $\C[ S_m]/K_{m,n}$.  \smallskip

This may be done using  Theorem \ref{dgoo} which is Theorem 3.3 of  \cite{note}.
\begin{definition}
Let $0<d$ be an integer and let $\sigma\in S_m$. 

Then $\sigma$ is called
{\em $d$--bad}\index{$d$--bad} if $\sigma$ has a descending subsequence of length $d$,
namely, if  there exists a sequence     $1\le i_1<i_2<\cdots <i_d\le n$ such that
$\sigma(i_1)>\sigma(i_2)>\cdots
>\sigma(i_d)$. Otherwise $\sigma$ is called {\em $d$--good}.

\end{definition}
\begin{remark}
$\sigma$ is $d$--good if any descending sub--sequence of $\sigma$ is of
length $\le d-1$. If $\sigma$ is $d$-good then $\sigma$ is $d'$-good for any
$d'\ge d$.

Every permutation is $1$-bad.
\end{remark}
\begin{theorem}\label{dgoo}The $n+1$--good permutations of $S_m$ form a basis of the quotient space $\C[ S_m]/K_{m,n}$.

\end{theorem}
\begin{proof}
See Theorem 3.3 of  \cite{note}.
\end{proof}  The irreducible representations $M_\lambda$ of the group $S_m$ are indexed by the partitions $\lambda\vdash m$ and the 
group algebra $\C[S_m]=\oplus_{\lambda\vdash m}M_\lambda\otimes M^*_\lambda$ as $S_m\times S_m$ representation.

The kernel of the map  $\C[S_m]\to End(V)^{\otimes m} $    is 0 if  $n\geq  m$ otherwise it is the ideal $K_{m,n}$ generated by the antisymmetrizer $\sum_{\sigma\in S_{n+1}} \epsilon_\sigma\sigma$ (where $ S_{n+1}\subset S_m$).

This ideal is  $$K_{m,n}= \oplus_{\lambda\vdash m\mid  ht(\lambda)>n}M_\lambda\otimes M^*_\lambda,\quad M_\lambda\simeq M^*_\lambda$$ where the height $ht(\lambda)$ of a partition is the number of its non zero rows.\medskip

Therefore the space of multilinear invariants in $\mathcal T$, although it is not an $S_m\times S_m$ representation, is linearly isomorphic (for all choices of $e,f\mid e+f=m$) to $  \oplus_{\lambda\vdash m\mid  ht(\lambda \leq n}M_\lambda\otimes M_\lambda$ and its dimension is 
\begin{equation}\label{laco}
c_m(M_n(F) ):=\sum_{\lambda\vdash m \mid ht(\lambda)\leq n }\chi_\lambda(1)^2
\end{equation} where $\chi_\lambda(1)$ is the dimension of the irreducible representation of $S_m$  associated to the partition $\lambda$ of $m$.

For multilinear matrix valued invariants we replace $m$ with $m+1$, Proposition \ref{ilmpsi} (2).\bigskip

As in the Theory of polynomial identities rather than discussing the dimension of $K_{m,n}$, which  grows as $m!$ it is better to study its codimension  $c_m(M_n(F) )$ in $\C[S_m]$, Formula \eqref{laco}, which equals  the  dimension of its image.

 \medskip

This sum $\sum_{\lambda\vdash m \mid ht(\lambda)\leq n }\chi_\lambda(1)^2$ is hard to compute, for $n=2$ it is given by   $\mathtt C_n:=\frac 1{n+1}\binom {2n}n $   the $n^{th}$ Catalan number.   The first Catalan numbers are  $$1, 1, 2, 5, 14, 42, 132, 429, 1430, 4862, 16796, 58786,\cdots$$

In general we have an asymptotic Formula due to Regev, see  \cite{agpr} Theorem 21.1.2.
\begin{theorem}\label{asmat1} (Regev) \begin{equation}
c_m(M_n(F) )\sim A_n m ^{\frac{1-n^2}2} n^{2(m+1)}. 
\end{equation}
where the constant 
\begin{equation}\label{assma0}
A_n=\Big[\Big(\frac1{\sqrt{2\pi}}\Big)^{n-1} \Big(\frac12 \Big)^{\frac{(n^2-1)}2}\cdot 1!2!\cdots (n-1)!\cdot n^{\frac{ n^2  }2} \Big].
\end{equation}    
\end{theorem}

\subsection{Symbolic method}\quad 

We want to prove
\begin{theorem}\label{phip}
We have, with the notations of Proposition \ref{ilmpsi} and $ev$ the evaluation of Theorem \ref{FFT}
\begin{equation}\label{php1}
\phi_\sigma =ev(\Phi_\sigma),\ \psi_\sigma =ev(\Psi_\sigma)\,.
\end{equation} This is an explicit formula  for the invariants.
\end{theorem}With the notations of Example \ref{unes} \begin{example}
$m=3+4=7,\ \sigma =  3,1,5,7,2,6,4=(1,3,5,2)(4,7)(6)$
$$\phi_\sigma= tr(\eta_1\eta_3\xi_2\eta_2)tr( \xi_1\xi_4)tr(\xi_3) $$
$$\psi_\sigma= tr(\eta_1\eta_3\xi_2\eta_2) tr(\xi_3)\xi_1 $$
$m=4+3=7,\ \sigma =  3,1,5,7,2,6,4=(1,3,5,2)(4,7)(6)$
$$\phi_\sigma=- tr(\eta_1\eta_3\xi_1\eta_2)tr( \eta_4\xi_3)tr(\xi_2) $$
$$\psi_\sigma= tr(\eta_1\eta_3\xi_1\eta_2) tr(\xi_2)\eta_4 $$
\end{example}
\begin{proof}
For bosonic matrices this is a classical result of Kostant, see  \cite{kostant} Lemma 4.9.\medskip

We can reduce to this case by the trick already used in Remark  \ref{alte}.

Add Grassmann variables  $\epsilon_i,\ i=1,\cdots,f$ and then  the matrices $\epsilon_i\xi_i$ have  entris in a commutative ring so Kostant's Formula specializes to
\end{proof}

 Kostant gives this.    Decompose $$\sigma^{-1}=(i_1,i_2,\cdots, i_{k_1})\cdots (i_{k+a },i_{k+a+1},\cdots,i_{k+a+r})$$ into cycles  $\mathtt c_1\mathtt c_2\cdots \mathtt c_r$ then if  $z_i$  are matrices oven a commutative ring we have
\begin{align}\label{Kof}&
  tr(  z_1\otimes\cdots \otimes z_m \circ \sigma^{-1})\\
 & = tr(z_{i_1},z_{i_2} \cdots z_{ i_{k_1}})\cdots tr(z_{i_{k+a }} z_{i_{k+a+1}}\cdots z_{i_{k+a+r} } ) 
\end{align} \begin{proposition}\label{emp}
We have $$ \phi_\sigma= ev(  \tau_{C_{e,f}}(\sigma ))= ev( \Phi_\sigma),\  \psi_\sigma= ev(  \tau_{C_{e,f}}(\bar\sigma ))= ev( \Psi_\sigma).$$ 
\end{proposition}\begin{proof} In Formulas \eqref{Kof}  substitute $z_i$ with $\eta_i$ for $i\leq e$  and $z_{e+i}$  with  $\epsilon_i \xi_i$.

The first line of Formula \eqref{Kof}  becomes $$\epsilon_1\cdots\epsilon_f (-1)^{f-1}tr(  \eta_1\otimes\cdots \otimes \eta_e\otimes \xi_1\otimes \cdots \otimes \xi_f  \circ \sigma^{-1})$$ 
The second line  of Formula \eqref{Kof} becomes  after rearranging the $\epsilon_i$ and introducing the sign 
$$\epsilon_1\cdots\epsilon_f (-1)^{f-1}ev(  \tau_{C_{e,f}}(\sigma ))\,.$$

%
%
%

\end{proof}
We have completed the proof of Theorem \ref{FFT}.\end{proof}

\section{Relations} By  Theorem \ref{gamm} the $T$--ideal of trace relations for the algebra $M_n(\mathbb K(V))$ is the kernel of the evaluation  of the free variables to the generic matrices $\eta_i,\ \xi_j$.

\medskip

 By Proposition \ref{ilmpsi} (1) Lemma \ref{iden} and  the definition of the evaluation in matrices, the multilinear trace identities  of  degrees $a,b$ in the space $\mathcal T(a,b)$ of  the trace algebra $\mathcal T$   are the image under the encoding  map $ \tau_{C_{a,b}}$   of the kernel $K_{m,n}$ of the map  of the group algebra $\C[S_m]$ to the space of endomorphisms of $(\C^n)^{\otimes m}$.

As a vector space  this is thus independent  of $e,f\mid e+f=m$. \medskip

Similar assertion for   the algebra $\mathcal E_n$. In this case the relations of degree $m$ are identified with the kernel $K_{m+1,n}$, Proposition \ref{ilmpsi} (2).

 \subsection{The \ch\ identities}
The theory follows in a way similar  to the classical case, reducing, by polarization and restitution,  to describe multilinear identities. \medskip

 In the classical  case  the space of multilinear relations of minimal degree is in degree $n$ for the algebra $\mathcal E_n$ and it is formed by the multiples of the polarized \ch \ identity defined:
\begin{equation}\label{CCH}
CH_n(y_1,\ldots,y_n)=\sum_{\sigma\in S_{n+1}}\epsilon_\sigma  \Psi_\sigma,\quad tr(CH_n (y_1,\ldots,y_n) y_{n+1})=\sum_{\sigma\in S_{n+1}}\epsilon_\sigma  \Phi_\sigma\,. 
\end{equation}
This follows from the vanishing of  the antisymmetrizer $A_{n+1}=\sum_{\sigma\in S_{n+1}}\epsilon_\sigma  \sigma$ as operator on $(\C^n)^{\otimes n+1}$.\smallskip

Then its interpretation by Formula  \eqref{sg}.  Therefore  the same Formula holds for all $e,f\mid e+f=n+1$, provided we remember   that the formula for $\phi_\sigma$ in the presence  of Grassmann variables  may contain a sign.
\begin{definition}\label{cah}
For any $e+f=n+1 $  we define:
\begin{align}\label{leeqb}
& CH_{e,f} (y_1,\ldots,y_e,x_1,\cdots,x_{f-1}):= \tau_{C_{e,f}}( \sum_{\sigma\in S_{n+1}}\epsilon_\sigma  \bar\sigma)=   \sum_{\sigma\in S_{n+1}}\epsilon_\sigma \Psi_ \sigma 
\\ 
&T_{e,f }(y_1,\ldots,y_e,x_1,\cdots, x_{f}):= \tau_{C_{e,f}}( \sum_{\sigma\in S_{n+1}}\epsilon_\sigma  \sigma) =   \sum_{\sigma\in S_{n+1}}\epsilon_\sigma \Phi_ \sigma  ,
\end{align}  
for $C_{e,f}$ the coloring $C_{e,f}(i)=0,\ i\leq e$ and  $C_{e,f}(i)=1,\ i>e$.   \end{definition}
\begin{remark}\label{conn} By Lemma \ref{labi} we have two possible identities:
\begin{align}\label{leeqb1}
&T_{e,f}(y_1,\ldots,y_e,x_1,\cdots,x_f,x_{f+1})=tr( CH_{e,f} (y_1,\ldots,y_e,x_1,\cdots,x_f)x_{f+1}))\\ 
&T_{e,f}(y_1,\ldots,y_e,x_1,\cdots ,x_{f+1}):=tr( CH_{e-1,f}  (y_1,\ldots,y_{e-1},x_1,\cdots,x_f,x_{f+1})y_{e})) 
\end{align}  In Definition \ref{cah} we have chosen the first.
\end{remark}
 \begin{theorem}\label{main1}
When evaluated in $n\times n$  matrices the two  expressions of Formula \eqref{leeqb} vanish.
\end{theorem}
\begin{proof}
For the second Formula  this follows from the fact that the operator $\sum_{\sigma\in S_{n+1}}\epsilon_\sigma  \sigma=0$ on $(\C^n)^{\otimes n+1}$. For the first one uses the fact that   
$$T_{e,f}(y_1,\ldots,y_e,x_1.\cdots,x_f,x_{f+1})=tr(CH_{e,f} (y_1,\ldots,y_e,x_1.\cdots,x_f)x_{f+1}) $$ since $x_{f+1}$ is a disjoint  variable this implies the claim.
\end{proof}  
For   $m=0+3=3$ and all  3 matrices fermionic
$$  \Phi_{ 1}=tr(x_1)tr(x_2)tr(x_3),\  \Phi_{ (1,2)}=tr(x_1x_2)tr(x_3), \Phi_{(2,3) }=tr(x_1)tr(x_2x_3),$$$$  \Phi_{ (1,3)}=-tr(x_1x_3)tr(x_2),\  \Phi_{ (1,2,3)}=tr(x_1x_2x_3)),\  \Phi_{(1, 3,2)}=-tr(x_1x_3x_2))$$
Then $\sum_{\sigma\in S_{3}}\epsilon_\sigma   \Phi_\sigma$ is $$  tr(x_1)tr(x_2)tr(x_3)-tr(x_1x_2)tr(x_3)-tr(x_1)tr(x_2x_3)$$$$+tr(x_1x_3)tr(x_2)+tr(x_1x_2x_3)  -tr(x_1x_3x_2) $$Then $\sum_{\sigma\in S_{3}}\epsilon_\sigma  \Psi_\sigma$ is
$$  tr(x_1)tr(x_2) -tr(x_1x_2)  -tr(x_1) x_2 +tr(x_2) x_1 + x_1x_2  - x_2x_1  $$   notice that in this case the expression for CH is  antisymmetric (so vanishes when putting $x_1=x_2$) while for ordinary matrices it is symmetric  and the polarization of the 1--variable \ch\ identity.

If $x_2$ is fermionic and $x_1=y_1$  bosonic then CH equals  the same formula of the ordinary case, 
$$  tr(y_1)tr(x_2) -tr(y_1x_2)  -tr(y_1) x_2 -tr(x_2) y_1 + y_1x_2  + x_2y_1.  $$   
The same Formula  which expresses  the relation in degree $n+1$ for the symmetric group, has $n+1$  different interpretations in the invariants for $n+1=e+f$  so we write these $n+1$ formulas as $CH_{e,f}(y_1,\ldots,y_e,x_1,\cdots,x_f)$.
\begin{remark}\label{simm}
$CH_{e,f}$ is symmetric in the bosonic and antisymmetric  in the  fermionic variables.  If we set equal to a single $y$ all the bosonic variables  we have a non zero element,  which for $f=0$  is  $n!$  times the usual 1--variable \ch\ identity.
\end{remark}
\subsection{The second fundamental Theorem}\quad We start with some combinatorics.\bigskip

Decompose $\underline m:=\{1,2,\cdots,m\}=A\cup B$ into two disjoint sets.  \smallskip

A coloring $C$ of $\underline m$ induces colorings  of $A,B$ so determines fermionic indices $A_f\subset A,\ B_f\subset B$ .

 \begin{lemma}\label{signs} Let $D,E$ two words in $A$ and $B$ respectively. Then
  \begin{equation}\label{lass}
\epsilon_C(D E)=\epsilon_C(AB)\epsilon_C(D  )\epsilon_C(  E) \,.
\end{equation}

\end{lemma}\begin{proof} When we reorder the word  $DE$  we can first reorder  $D,E$  separately, getting the sign $\epsilon_C(D  )\epsilon_C(  E)$ and then $AB$   getting the sign $        \epsilon_C(AB)$.  

Thus 
  Formula \eqref{lass}.
\end{proof} 
\begin{remark}\label{due}
For each $j\in B_f$ let $|j|$ the number of $i\in A_f$ with $i>j$. 
 
 For each $i\in A_f$ let $|i|$ be the number of $j\in B_f$ with $i>j$. 
    \begin{equation}\label{glii}
\epsilon_C(AB)=\prod_{j\in B_f}(-1)^{|j|}=\prod_{i\in A_f}(-1)^{|i|}\,.
\end{equation} 
\end{remark}

Let us denote   as in  page \pageref{ze}, for all $h$,      $z_h=y_h$  if $ C(h)=0$ and $z_h=x_h$  if $ C(h)=1$.  

  \begin{lemma}\label{ful}
  Let $A,B$ and $C$ be as before.  Consider a permutation $\sigma$ of $A$, an index $i\in A$  and a string  $E$  permutation of $B$. 
  
  Let $\mathtt c:=(i,E)$ a cycle. Then if $C(i)\neq C(i,E)$ change the color  of $i$ and call the new coloring $C'$. Otherwise $C=C'$. We have:
  \begin{equation}\label{mm}
 \tau_C(\sigma\circ \mathtt c)=\gamma \cdot \tau_{C'}(\sigma)|_{z_i\mapsto  \tau_C((i,E))}
\end{equation}
eith $\gamma=\pm 1$ a sign dependent only on $A,B,i,C$ and NOT on $\sigma$.
\end{lemma}
 
\begin{proof}

 Write  $\sigma=  \mathtt c_1\circ\cdots\circ\mathtt c_r\circ (\cdots,i)$  a p-string with support a string  $F,i $, so the support of $F$ is $A\setminus\{i\}$. Then we have the cycle  decomposition of $\sigma\circ \mathtt c$:
$$\sigma\circ \mathtt c=  \mathtt c_1\circ\cdots\circ\mathtt c_r\circ (\cdots,i,E)\,. $$
  By Formula  \eqref{tau} $ \tau_C(\sigma\circ \mathtt c)$  is the trace monomial obtained by the substitutions $h\mapsto z_h$  in the p-string of support  $F,i,E$  obtaining a trace monomial $M$ times the sign $\epsilon_ C(F,i,E)$.

\medskip

Now   $ \tau_{C'}(\sigma)$ is the trace monomial obtained by the substitutions $h\mapsto z_h$  in the p-string of support  $F,i $ times the sign $\epsilon_ {C'}(F,i )$.\medskip

By the hypotheses on $C'$  in $ \tau_{C'}(\sigma)$ we can perform the substitution $ \tau_{C'}(\sigma)|_{z_i\mapsto  \tau_C((i,E))}$  (cf. Definition  \ref{tid}) obtaining the same  trace monomial $M$ of $ \tau_C(\sigma\circ \mathtt c)$ times the sign $\epsilon_ {C'}(F,i )\epsilon_ {C}( i ,E)$; hence   we have that
$$\gamma=  \epsilon_ C(F,i,E)\epsilon_ {C'}(F,i )\epsilon_ {C}( i ,E)$$
By Lemma  \ref{signs} and Formula \eqref{lass}    for $D=F,i$.
\begin{equation}\label{efi}
\epsilon_C(F,i ;E)=\epsilon_C(AB)\epsilon_C(F,i  )\epsilon_C(  E)
\end{equation} Hence
$$\gamma= \epsilon_C(AB)\epsilon_C(  E)\epsilon_ {C}( i ,E)\boxed{\epsilon_C(F,i  )\epsilon_ {C'}(F,i )}\,.$$
If $C=C'$ we have $\epsilon_C(F,i  )\epsilon_ {C'}(F,i )=1$ otherwise in one of these two factors the index $i$ is bosonic (say for instance for the color $C$) while in the other it is fermionic.  

 Then $\epsilon_C(F,i  )=\epsilon_C(F  )$  while  $\epsilon_ {C'}(F,i  )=(-1)^{|i|}\epsilon_ {C'}(F  )$ with $|i|$  the number of indices $j$  in $A\setminus\{i\}$  with $j>i$.\smallskip

Therefore, since on the support  $A\setminus\{i\}$   of $F$ the two colors coincide,  we have $\epsilon_C(F  )=\epsilon_ {C'}(F  )$:$$\gamma= \begin{cases}\epsilon_C(AB)\epsilon_C(  E)\epsilon_ {C}( i ,E)\quad\text{if}\quad C=C'\\
(-1)^{|i}\epsilon_C(AB)\epsilon_C(  E)\epsilon_ {C}( i ,E)\quad\text{if}\quad C\neq C'\,.
\end{cases} $$
\end{proof}

In order to proceed we need, for $p\geq m$:\begin{lemma}\label{splici}
Every permutation $\gamma$ in $S_{p}$ can be written as a product
$\gamma=\alpha\circ\beta$ where $\alpha\in S_{m}$ and, in each cycle of $\beta,$ there is at most 1 element in
$1,2,\dots,m$.
\end{lemma} 

\begin{proof} It is enough to do it when  $\gamma$ is a cycle.

Observe  that, for $a,b,c\ldots$ numbers $\leq m$ and $A,B,C,\ldots$ strings of numbers  $> m$,   we have:
 \begin{equation}\label{sppli}
\gamma=(a\ A\ b\ B\ c\ C\  \ \dots\  e\ E)=   (a\ b\ c\dots) (A\ a)( B\ b)(C\ c)\dots (E\ e).
\end{equation}
\end{proof}  
  \begin{example} If $A=7,5,4 ;\ B=6,3;\ a=1,\ b=2$, $$ (1,7,5,4,2,6,3)=(1,2) (7,5,4,1) (6,3,2) .$$\end{example}
  
We are now ready to prove the second fundamental theorem, we use the notations  $\mathcal T_n$ and $\mathcal E_n$ as in Theorem \ref{FFT}  for the algebra  of invariants  of  both bosonic and fermionic and  
 the algebra  of equivariant  maps  of  both bosonic and fermionic matrices.  We call the relations in $\mathcal T_n$  {\em trace relations} and in $\mathcal E_n$  {\em trace identities}.

\begin{theorem}\label{trnm}[SFT] i) The  T-ideal of trace relations, kernel of the evaluation map of $\mathcal T$ into $\mathcal T_n$  is generated (as
a T-ideal) by the trace $n+1$ trace relations
$T_{e,f}$ with $e+f=n+1$.

ii)  The T-ideal of (trace) identities of $n\times n$ matrices,    is generated (as a T-ideal) by the $n+1$   Cayley Hamilton identities $CH_{e,f}(x,y),\ e+f=n+1$,  and $t(1)=n$, Formula \eqref{leeqb}.\end{theorem}

\begin{proof} i) From   the  Aronhold method every relation is a sum of multilihomogeneous relations which can be fully polarized to multilinear trace relations. \smallskip

 The original relation is then obtained up to a multiplicative integer factor  by the process of restitution, that is making the variables  equal in each group.\smallskip

It is  thus sufficient to prove that a multilinear trace relation (resp. identity)  is in the T-ideal generated by the $T_{e,f}$ (resp. the $CH_{e,f}$).

{\em Here by  multilinear  we mean linear in some subset of the variables, without loss of generality we may assume  the first variables.}\smallskip  

 Let us first look at trace
relations. By Proposition \ref{ilmpsi} (4), Lemma \ref{iden} and  the definition of the evaluation in matrices, the multilinear trace identities  of  degrees $a,b$ in the space $\mathcal T(a,b)$ of  the trace algebra $\mathcal T$   are the image under the encoding  map $ \tau_{C_{a,b}}$   of the kernel $K_{m,n}$ of the map  of the group algebra $\C[S_m]$  in the space of endomorphisms of $(\C^n)^{\otimes m}$. 
\bigskip

By the classical theory  this kernel is 0 if $m\leq n$, otherwise it is the two sided ideal of  the group algebra $ \C[S_{m}]$  generated by the antisymmetrizer $A_{n+1}\in \C[S_{n+1}]\subset  \C[S_{m}]$. 

In particular there are no trace relations in degree $\leq n$.\smallskip

 We thus have to decode  this Kernel  as elements of  $\mathcal T(a,b)$ under the encoding map.\medskip

 The space $K_{m,n}$ is linearly spanned by elements  $ \tau(\sum_{\sigma\in S_{n+1}}\epsilon_\sigma \sigma)\beta ,$ with $ \tau,\beta\in S_{m}$.  Hence it is enough to look at the relations of the type
$$R= \tau_{C_{a,b}}(\tau(\sum_{\sigma\in S_{n+1}}\epsilon_\sigma \sigma)\beta),\ \tau,\beta\in S_{m}.$$ \smallskip

 We write such a
relation as
 $ \tau_{C_{a,b}}(\tau(\sum_{\sigma\in S_{n+1}}\epsilon_\sigma \sigma)  \tau^{-1} \tau\beta)$. \smallskip

 We set $A:=\tau(\{1,2,\dots,n+1\})$ and  $B$ its complement in $\underline m$.\smallskip

By Lemma \ref{splici} we can write $ \tau\beta =\alpha\theta$ with  $\alpha\in S_{A}$ and in each cycle of $\theta $ there is at most 1 element in  $A$. We have
$$ \tau(\sum_{\sigma\in S_{n+1}}\epsilon_\sigma \sigma)  \tau^{-1}\alpha =\pm  \tau(\sum_{\sigma\in S_{n+1}}\epsilon_\sigma \sigma)  \tau^{-1}$$ so we are reduced to study a relation
$$ \tau_{C_{e,f} }(\tau(\sum_{\sigma\in S_{n+1}}\epsilon_\sigma \sigma)  \tau^{-1}\theta)  $$in each cycle of $\theta $ there is at most 1 element in  $A$.\medskip

We apply in iteration  Lemma \ref{ful} and see  that $R$  is deduced from  the relation $ \tau_{C }(\tau(\sum_{\sigma\in S_{n+1}}\epsilon_\sigma \sigma)  \tau^{-1}$,  for some color  computed recursively  from the color $C_{a,b }$  restricted to $A$, by a sequence of substitutions of  the variables corresponding to the  indices in $A$ of the cycles of $\theta$.\smallskip

Since $ \tau_{C }(\tau(\sum_{\sigma\in S_{n+1}}\epsilon_\sigma \sigma)  \tau^{-1})$ is obtained from one of the trace relations of  Formula \eqref{leeqb}  by a substitution of variables   i) is proved.
\medskip

ii)\  Consider now a multilinear   relation  $R$  in $\mathcal F_T$ (trace identity) which we may assume to be  in the variables  $y_1,\ldots,y_a,x_1.\cdots,x_b$.\smallskip

 Then 
the formula $tr(R x_{b+1}  )$  is a    multilinear trace   relation  $R$  in $\mathcal T$  in the variables  $y_1,\ldots,y_a,x_1.\cdots,x_b,x_{b+1} $ so, from part i),   there are no trace identities in degree $\leq n-1$.\smallskip

If $a+b=n$ then there $\dim K_{n+1,n}=1$  so there is only the relation $CH_{a,b}$.

If $a+b> n$ we have that $R$   is a linear combination   of terms of types \begin{equation}\label{moo}
T=T_{e,f}(M_1,\ldots,M_e,N_1,\cdots,N_f,N_{f+1})tr(P_1)\cdots tr(P_r)
\end{equation} with  the $M_i$ monomials of degree 0, the $N_i$ monomials of degree 1 and the $P_i$ monomials of arbitrary degree.\smallskip

We can write  each $T=tr(\mathtt f\cdot x_{b+1}  )$  for an element $\mathtt f\in\mathcal F_T$  vanishing on $n\times n$ matrices and $R$ is then a linear combination of these elements $\mathtt f$.\smallskip

The Formula for $\mathtt f$  depends upon where $x_{b+1}$ appears in the Formula \eqref{moo}.

If it appears in one of the $P_i=Ax_{b+1}B$  then  up to a sign $\mathtt f=BA$ times the  element $T$ with $P_i$  omitted.

Otherwise up to some reordering, cf. \ref{simm}, and sign we may assume that we have $N_{f+1}=Ax_{b+1}B$ with $A,B$  monomials or  $M_{e}=Ax_{b+1}B$ with $A,B$  monomials.  Then we see that $\mathtt f$ is up to sign 
$$\mathtt f=B \cdot C_{e,f}(M_1,\ldots,M_e,N_1,\cdots,N_f )A\cdot tr(P_1)\cdots tr(P_r) $$ or 
$$\mathtt f=B\cdot  C_{e-1,f+1}(M_1,\ldots,M_{e-1},N_1,\cdots,N_f,N_{f+1} )A\cdot tr(P_1)\cdots tr(P_r)\,. $$ 
The proof is complete.\end{proof}  In \S \ref{unf}  we have seen that for $n=1$  we have two relations or identities,  the  Amitsur--Levitzi equation  $\xi^{2n}=0$ and equation \eqref{tru}.

According to the Theorem just proved they should be deducible from the list $CH_{e,f},\ e+f=n,\  $ Formulas \eqref{leeqb}. We suggest the reader to do this as Exercise.
\begin{exercise}
$\xi^{2n}=0$ is deduced from $CH_{n,0}$  by replacing all the bosonic variables with $\xi^{2 } $.\smallskip

Equation \eqref{tru}  is deduced from $CH_{n-1,1}$  by replacing all the bosonic variables with $\xi^{2 } $ and the single fermionic variable with $\xi$.
\end{exercise}
Both the first and second fundamental theorem for matrices are not as precise as the
corresponding theorems for vectors and forms.  

Several things are lacking;   a description of a minimal set
of generators for the trace invariants and a description of a minimal set
of generators of the algebra $\mathcal E_n\langle X,Y\rangle$  of equivariant maps from $m$--tuples of matrices to matrices  as   module over the ring on invariants. \medskip

The results are identical to the classical results for which at  the moment the best result is given by the following:
\begin{theorem}[Razmyslov] \label{fgrm}The invariants of $n\times n$ matrices are generated by the elements
$tr(M)$ where $M$ is a monomial of degree $\leq n^2$.\smallskip

The algebra  $\mathcal E_n\langle X,Y\rangle$  of equivariant maps from $m$--tuples of matrices to matrices is spanned as a module over the ring on invariants, by monomials of degree $<n^2$.
\end{theorem}
It is conjectured that the best estimate should be  $<\binom n2$. Kuzmin has proved that this is a lower bound, \cite{Kuz}.

\subsection{The Poincar\'e series  and cocharacters\label{Picoc}}

We consider $\N$ multi graded vector spaces that is a vector space $V=\oplus_{\underline n\in \N^k}V_{\underline n}$. 

Here $\underline n=(n_1,n_2,\cdots,n_k)$  where $k$ could also be $\infty$ but then we assume that $\underline n$ has finite support.\medskip

We further assume that  $\dim V_{\underline n}<\infty $ for all $\underline n$ and then we write this information in a series
\begin{equation}\label{pos}
\mathcal P(V)=\sum_{\underline n\in \N^k}\dim V_{\underline n}\,t^{\underline n}\,.
\end{equation}
Where $t$ is a list of variables $t_1,\cdots, t_k$ and $t^{\underline n}=\prod_it_i^{n_i}$.\medskip

We call this the {\em multigraded Hilbert--Poincar\`e series of $V$}.  This series  has the simple properties, for two graded vector spaces $V_1,V_2$:
\begin{equation}\label{supp}
\mathcal P(V_1\oplus V_2)=\mathcal P(V_1)+\mathcal P(V_2),\  \mathcal P(V_1\otimes V_2)=\mathcal P(V_1)\mathcal P(V_2)\,.
\end{equation}
It also has a nice expression under the construction of symmetric or exterior algebra.

For a vector  $v$ of multidegree $\underline n=(n_1,n_2,\cdots,n_k)$ we have $S(Fv)$ has basis  $v^m,\ m=0,\infty$ so $\mathcal P(S(Fv))=\frac1{1-t^{\underline n}}$  and $\bigwedge Fv $ has basis  $1,\ v $  so $\mathcal P(\bigwedge Fv) = 1+t^{\underline n}.$

In our treatment we want to introduce also fermionic indices so we consider a multigraded  supervector space $V=\oplus_{\underline n,\underline m\in \N^k}V_{\underline n,\underline m}$.

 We then assign $\Z/(2)$ degree 0  to the spaces for which $|\underline m|=\sum_im_i$ is even and 1 if it is odd.\smallskip

In the paper \cite{zm} the authors introduce also an {\em index},  if one has both bosonic and fermionic indices  then one can speak of fermionic space when the sum of the fermionic indices is odd.  Then one  assigns to such a space the negative dimension. 
Taking account the sign we have a super Poincar\'e series, called {\em index} in \cite{zm}  and given by:

 \begin{equation}\label{pos1}
\mathcal I(V)=\sum_{\underline n\in \N^k}(-1)^{|\underline m|}\dim V_{\underline n,\underline m}\,t^{\underline n}u^{\underline m}
\end{equation} 
The analogous of Formula \eqref{supp}  holds for the index.\medskip

One way of forming this is the following, iven a  multi graded vector space $V$ we double it as $V_0\oplus V_1$  and form  $\mathbb K[  V_0\oplus V_1]=S(V_0)\otimes \bigwedge V_1$ and we have, introducing variables $u_i$ for the fermionic indices.
\begin{align}\label{pok}
\mathcal P(\mathbb K[  V_0\oplus V_1])=\prod_{\underline n} \frac1{(1-t^{\underline n})^{\dim V_{\underline n}}}\prod_{\underline n}  (1+u^{\underline n})^{\dim V_{\underline n}}\\
\mathcal I(\mathbb K[  V_0\oplus V_1])=\prod_{\underline n} \frac1{(1-t^{\underline n})^{\dim V_{\underline n}}}\prod_{\underline n}  (1+(-u)^{\underline n})^{\dim V_{\underline n}}\,.
\end{align}
Setting $u=t$ one has the remarkable identity $\mathcal I(\mathbb K[  V_0\oplus V_1])[t,t]=1$.\medskip

The previous formula for index  is valid also for the invariants, where we do not have a very explicit formula  for the series. 

In fact this is a general fact  for $V$ and graded representation of a compact Lie group and follows from the Weyl--Molien Formula.\smallskip

If  $V$ is graded representation of a compact Lie group $L$  one writes the character of the representation restricted to a maximal torus $T$ (in case of $U(N)$ the diagonal matrices).   Then each representation $ V_{\underline n}$ has a basis $v_1,\cdots, v_m$ of eigenvectors for $T$  with eigenvalues  characters $\chi(i)$ of $T$.

Then one can incorporate the character in the Poincar\'e series replacing $(1-t^{\underline n})^{\dim V_{\underline n}}$ with 
$\prod_{i=1}^{\dim V_{\underline n}}(1-\chi(i)t^{\underline n})$
\begin{equation}\label{pok2}
\mathcal P_\chi(\mathbb K[  V_0\oplus V_1])=\prod_{\underline n} \frac1{\prod_{i=1}^{\dim V_{\underline n}}(1-\chi(i)t^{\underline n})}\prod_{\underline n} \prod_{i=1}^{\dim V_{\underline n}} (1+\chi(i)u^{\underline n}) ) \,.
\end{equation} Then for the ring of invariants one has to integrate this series, as function on $T$ via the characters,  times the Weyl factor $\Delta\bar\Delta$ and divided by the order of the Weyl group $|W|$  \begin{equation}\label{pok3}
\mathcal P(\mathbb K[  V_0\oplus V_1]^L)=\frac 1{|W|}\int_T\Delta\bar\Delta \prod_{\underline n} \frac1{\prod_{i=1}^{\dim V_{\underline n}}(1-\chi(i)t^{\underline n})}\prod_{\underline n} \prod_{i=1}^{\dim V_{\underline n}} (1+\chi(i)u^{\underline n}) ) d\nu\,.
\end{equation}
With $d\nu$ the invariant measure of volume 1 on $T$.\medskip

\begin{theorem}\label{ind}
Setting $t=u$ and changing the signs for the fermionic characters one   gets 1 for the index $\mathcal I(\mathbb K[  V_0\oplus V_1]^L)=1$.
\end{theorem}\bigskip

In the case of $\mathbb K(V^*)=S[\eta_{h,k}^{(i)}]\otimes \bigwedge[\xi_{h,k}^{(i)}] $ we have
\begin{equation}\label{MoW0}
\mathcal P_\chi(\mathbb K(V^*))= \frac{ \prod_{k=1}^\infty\prod_{i,j=1}^n(1+z_i^{-1}z_ju_k)}{\prod_{k=1}^\infty\prod_{i,j=1}^n(1-z_i^{-1}z_jt_k)} 
\end{equation}

%
The Molien--Weyl Formula   gives:
\begin{equation}\label{MoW1}
\mathcal P(\mathbb K(V^*)^{U(n)})=\frac 1{n!}\int_T\frac{\prod_{i\neq j}(1-z_i^{-1}z_j)\prod_{k=1}^\infty\prod_{i,j=1}^n(1+z_i^{-1}z_ju_k)}{\prod_{k=1}^\infty\prod_{i,j=1}^n(1-z_i^{-1}z_jt_k)}d\nu
\end{equation}
$$T=\{(z_1,\cdots,z_n)\mid |z_i|=1\},\qquad d\nu=\frac 1{(2\pi\ii)^n}\frac{dz_1\wedge\cdots \wedge dz_n}{z_1\cdots z_n}\,.$$
Notice that in the language of \cite{zm} the variables $t_k,u_k$  can be replaced  by powers   $t_k\mapsto q^{c(k)},\  u_k\mapsto q^{d(k)},\ c(k),\ d(k)\in\N$ are some {\em charges} in particular $c(k)=d(k)=k$.

Then the Poincar\'e series becomes a series in $q$  where the coefficient of $q^\ell$ is the dimension of the space of elements of charge $\ell$ (which is formed by elements of different degrees).

 By Theorem \ref{ind} as soon as the bosonic variables have the same charge as the fermionic variables the index vanishes.
 
 \subsubsection{Comparing with \cite{zm}}
 
 In this paper the Formulas are written in a different form and for purely fermionic matrices $V^*_1$ they prove the remarkable surprising fact (\S 3 Rank--invariance of $\frac 1
4$--BPS index)  that 
 \begin{theorem}\label{inden}
The index is independent of $n$ provided we replace the variable  $u_k$ with the charge $q^k$. Precisely the integral
 \begin{equation}\label{MoW4}
\mathcal I(\mathbb K(V^*_1)^{U(n)})=\frac 1{n!}\int_T  \prod_{i\neq j}(1-z_i^{-1}z_j)\prod_{k=1}^\infty\prod_{i,j=1}^n(1-z_i^{-1}z_jq^k)d\nu
=\prod_{k=1}^\infty(1-q^k)\end{equation} 
\end{theorem} Notice that  $\prod_{i,j=1}^n(1-z_i^{-1}z_jq^k)=(1-q^k)^n\prod_{i\neq j=1}^n(1-z_i^{-1}z_jq^k)$ and $(1-q^k)^{n-1}\prod_{i\neq j=1}^n(1-z_i^{-1}z_jq^k)$ is the character of trace 0 matrices.\medskip

 This identity should be understood as follow.  Replace $V^*$  with $V^*_{1,0}$ the trace 0  fermionic  matrices then one has $\mathcal I(\mathbb K(V^*_{1,0})^{U(n)})= 1$. This means 
 \begin{proposition}\label{trz} i)\quad $\mathcal I(\mathbb K(V^*_0)^{U(n)})= 1$.\smallskip

ii)\quad For each charge $\ell$ the vector space of bosonic invariants (respectively relations) for trace 0 matrices of charge $\ell$ has the same dimension as the the vector space of fermionic  invariants (respectively relations).
\end{proposition}   
\begin{proof}
Decompose  $V_1^*=T\oplus V_{1,0}^*$  with $T$ the vector space with basis  the invariant elements $tr(\xi_i)$.\smallskip
 
Then $\mathbb K(V^*_1)= \mathbb K(T)\otimes  \mathbb K(V^*_{1,0})$ and, since $\mathbb K(T)$ is an invariant algebra we have
$$ \mathbb K(V^*_1)^{U(n)}\simeq \mathbb K(T)\otimes \mathbb K(V^*_{1,0})^{U(n)}\implies  \mathcal I(\mathbb K(V^*_1)^{U(n)})= \mathcal I( \mathbb K(T)) \mathcal I( \mathbb K(V^*_{1,0})^{U(n)})\,.$$
Now $\mathcal I( \mathbb K(T))   =\prod_{k=1}^\infty(1-q^k)$  hence the claim of i).\smallskip

Part i) implies that for each charge $\ell$ the vector space of bosonic invariants   for trace 0 matrices of charge $\ell$ has the same dimension as the the vector space of fermionic  invariants. In particular  for $n=\infty$  there are no relations so that the space of  bosonic  formulas   for trace 0 matrices of charge $\ell$ has the same dimension as the the vector space of fermionic   formulas, so the claim for relations follows.
\end{proof}
This proposition says nothing about how to establish a linear isomorphism between  the two spaces for a given charge $\ell$.

In fact I suspect that there is no natural  isomorphism. This is quite mysterious even for the formulas.
\begin{example}\label{nd}
Foe $n=2$  the first relations appear with charge 6. A fermionic  $T_{0,3}(x_1,x_2,x_3)$ in degree 3 and the bosonic $T_{1,2}(x_1^2,x_1 ,x_2)$ of degree 4.

There is no apparent relation between the two.
\end{example}
 One starts with the Formula
 $$\frac 1{1-t}=\exp(-\log(1-t))=e^{ \sum_{k=1}^\infty \frac 1k t ^k} $$
 then 
 $$\prod_{i,j=1}^n(1-z_i^{-1}z_jt_k))^{-1}= e^{\sum_{h=1}^\infty \frac 1h \sum_{i,j=1}^n(z_i^{-1}z_j)^ht_k ^h} $$
 So if $U$ is a matrix with eigenvalues  $z_i$  we have $\sum_{i,j=1}^n(z_i^{-1}z_j)^h=tr(U^h) tr(U^{-h}).$

 Finally 
 \begin{equation}\label{fors}
\frac 1{\prod_{k=1}^\infty\prod_{i,j=1}^n(1-z_i^{-1}z_jt_k)}=e^{\sum_{h=1}^\infty \frac 1h(\sum_{k=1}^\infty t_k ^h) tr(U^h) tr(U^{-h})} 
\end{equation}
 Of course $tr(U^h) tr(U^{-h})=tr((U\otimes U^{-1})^h)$ and $U\otimes U^{-1}$ is the matrix of $U$ in the conjugation representation.
 
Similarly for fermionic matrices where the formula to start from is. 
 $$ 1+t =\exp( \log(1+t))=e^{- \sum_{k=1}^\infty \frac 1k(- t )^k} $$
Then the Poincar\'e series for the invariants of bosonic and fermionic matrices is obtained by integration
 \begin{equation}\label{mas}
 \int_{U(n)}e^{\sum_{h=1}^\infty \frac 1h(\sum_{k=1}^\infty (t_k ^h-(-1)^hu_k^h)) tr(U^h) tr(U^{-h})} du
\end{equation}with $du$ normalyzed Haar measure. From this by Weyl integration formula one has Formula \eqref{MoW1}.
 
 A similar manipulation could be performed to Formula \eqref{pok2} but we leave this.\medskip
 
 In his paper   \cite{sm} Sameer Murthy  transforms formula \eqref{mas} in a very interesting more computable form. Let us denote by $d(h):=(\sum_{k=1}^\infty (t_k ^h-(-1)^hu_k^h))$ then expanding the exponential in power series one has
 \begin{equation}\label{mas1}
 \int_{U(n)} \sum_{j=0}^\infty \frac 1 {j!}\left(\sum_{h=1}^\infty \frac{d(h)}h tr(U^h) tr(U^{-h})\right )^j du
\end{equation}
Then compute the $j^{th}$ term as a sum over  partitions $\lambda\vdash j:\ 1^{a_1}2^{a_2}\cdots r^{a_r},$  i.e. $j=\sum a_ii$ setting  $\prod_i        tr(U^{ i})^{a_i}:=tr_\lambda(U)$.
 $$\sum_{\lambda\vdash j}\frac 1{1^{a_1}2^{a_2}\cdots r^{a_r}}\binom j{1^{a_1},2^{a_2},\cdots, r^{a_r}}\prod d(i)^{a_i}  \int_{U(n)}tr_\lambda(U)\overline{tr_\lambda(U)}du$$
 Then then has the usual Frobenius Formula  
 $$tr_\lambda(U)=\sum_{\mu\vdash j}\chi_\mu(\lambda) S_\mu(U)\,.$$
 Where $S_\mu(U)$ is the character of the representation of $U(n)$ associated to  $\mu$.  For $U(n)$  one has the bound $ht(\mu)\leq n$ on the height of $\mu$.\smallskip
 
 Using the orthonormality of the characters of  $U(n)$  one has finally  
 $$\int_{U(n)}tr_\lambda(U)\overline{tr_\lambda(U)}du=\sum_{\mu\vdash j\mid ht(\mu)\leq n}\chi_\mu(\lambda) ^2\,.$$
 
 Formula \eqref{mas} becomes thus, setting $$c(\lambda):=\frac 1{1^{a_1}2^{a_2}\cdots r^{a_r}}\binom j{1^{a_1},2^{a_2},\cdots, r^{a_r}},\ d(\lambda):=\prod d(i)^{a_i} $$
 \begin{equation}\label{mas2}\text{Formula \eqref{mas}}\ =
\sum_{j=0}^\infty \frac 1 {j!}\sum_{\lambda\vdash j} c(\lambda) d(\lambda) \sum_{\mu\vdash j\mid ht(\mu)\leq n}\chi_\mu(\lambda) ^2\,.
\end{equation}
 \medskip

\subsubsection{A different approach}\quad  Formulas are written using the the $q$-Pochhammer symbol $$(x; q)_n =
\prod^{n-1}
_{i=0} (1-x q^i)\implies (q;q)_\infty=\prod^{\infty}
_{i=1} (1- q^i).$$
 
 The computation in \cite{zm}  is the following:
 
 \begin{align*} 
&\mathcal I(\mathbb K(V_1^*)^{U(n)})=\frac 1{n!}\int_T {\prod_{i\neq j}(1-z_i^{-1}z_j)\prod_{k=1}^\infty\prod_{i,j=1}^n(1-z_i^{-1}z_jq^k )} d\nu
\\&=\frac 1{n!}\int_T { \prod_{k=1}^\infty(1-q^k)^n\prod_{i\neq j }^n  \prod_{k=0}^\infty (1-z_i^{-1}z_jq^k )} d\nu\\&=\frac {(q;q)_\infty^n}{n!}\int_T {  \prod_{i\neq j }^n  \prod_{k=0}^\infty (1-z_i^{-1}z_jq^k )} d\nu\,.
\end{align*}
 
 So one needs to compute the constant coefficient of
\begin{equation}\label{mani}
\frac 1{n!}\prod_{i\neq j }^n  \prod_{k=0}^\infty (1-z_i^{-1}z_jq^k )=\frac 1{n!}\prod_{i\neq j }^n (z_i^{-1}z_j;q)_\infty
\end{equation}
 which, is reduced in \cite{zm} to   a   deep combinatorial Theorem and equals $ (q;q)_\infty^{1-n}$  so that finally Formula \eqref{MoW4} holds.\smallskip
 
 The combinatorial Theorem was conjectured by Andrews, as a $q$--analogue of a conjecture by Dyson of 1962, and proved by Doron  Zailberger 
 and 
David M. Bressoud  in 1984 \cite{bz}. In this paper one can also find interesting historical background.\smallskip
 
 Their Theorem in fact is  more general  but in our case it is 
 
 \begin{equation}\label{andr}
\prod_{i< j }^n (z_i^{-1}z_j;q)_\infty\prod_{i< j }^n (qz_j^{-1}z_i;q)_\infty\stackrel{\text{constant coefficient}}\mapsto (q;q)_\infty^{1-n}\,.
\end{equation} we need to use this for computing  \eqref{mani}.  \smallskip

The clever proof in \cite{zm} goes as follows.  Notice that Formula \eqref{andr} can be written as $$\prod_{i< j }^n(1- z_i^{-1}z_j)f(z),\quad\text{with}\quad f(z)=\prod_{i< j }^n (qz_i^{-1}z_j;q)_\infty\prod_{i< j }^n (qz_j^{-1}z_i;q)_\infty$$ a symmetric function. Therefore the constant term is also the constant term of its average
$$ \frac 1{n!}\sum_{\sigma\in S_n}\prod_{i< j }^n(1- z_{\sigma(i)}^{-1}z_{\sigma(j)})f(z)$$
next they observe that 
$$\prod_{i< j }^n(1- z_i^{-1}z_j)=\prod_iz_i^{n-i}\prod_{i< j }^n(z_i- z_j) =\prod_iz_i^{n-i}V(z_1,\cdots, z_n)$$ with $V(z_1,\cdots, z_n)$ the Vandermonde determinant so
$$ \sum_{\sigma\in S_n}\prod_{i< j }^n(1- z_{\sigma(i)}^{-1}z_{\sigma(j)})= [\sum_{\sigma\in S_n}\epsilon_\sigma \prod_iz_{\sigma(i)}^{n-i}]V(z_1,\cdots, z_n) $$$$=V(z_1^{-1},\cdots, z_n^{-1})V(z_1,\cdots, z_n)=\prod_{i\neq j }^n(1- z_i^{-1}z_j)\,.$$ Then
 \begin{equation}\label{mani1}\frac 1{n!}\sum_{\sigma\in S_n}\prod_{i< j }^n(1- z_{\sigma(i)}^{-1}z_{\sigma(j)})f(z)=\frac 1{n!}\prod_{i\neq j }^n(1- z_i^{-1}z_j)f(z)
\end{equation}
which is Formula \eqref{mani} and the computation follows.
\bigskip

A special case of Proposition \ref{trz} is the following.\begin{proposition}\label{trz}   
 
  For each charge $\ell$ the vector space $\mathcal T_\ell^b$ of bosonic elements in the free algebra $\mathcal T$  for trace 0 matrices of charge $\ell$ has the same dimension as the the vector space of fermionic  elements $\mathcal T_\ell^b$.
\end{proposition}   
I do not see any  a priori explanation of this  which looks as a combinatorial statement.  

It should be possible to find explicit formulas. It is easy to see that the first non zero case is for $\ell=3$. The beginning of the Poicar\'e series  for  $\mathcal T_\ell^b$ is
$$1+q^3+q^4+3q^5+6q^6+{11}q^7+\cdots$$
%
\medskip

Computing the invariants of $3\times 3$ traceless fermionic matrices one gets that for charge 7 the two spaces of bosonic or fermionic invariants have dimension 10 and not 11 as for the free case, see next table. 

 In fact for fermionic matrices we have the relation $t(x_1^7)=0$ while for the bosonic  we have $\mathcal T_{2,2}(x_1^2,x_1^2,x_1,x_2 )=0.$ Explicitely of the 24 terms indexed by permutations  of $\mathcal T_{2,2}(x_1^2,x_1^2,x_1,x_2 )$ all permutations with a fixed point give 0.  Remain the 6 four cycles which give $6 tr(x_1^5x_2)$  and the 3 cycle of type 2,2, which give
$$ tr(x_1x_2)tr(x_1^4)=0,
2  tr(x_1^2x_2)tr(x_1^3)\implies 3tr(x_1^5x_2)+ tr(x_1^2x_2)tr(x_1^3)=0.$$
This can ne also deduced from \eqref{tru} by muliplication by $x_2$ and taking the trace.
\begin{remark}\label{intr}
It would be interesting to establish an explicit combinatorial bijection between bosonic and fermionic trace monomials with the same charge.

I do not see a pattern in the formulas  just displayed for $\ell\leq 7$.
\end{remark}

\bigskip

\bigskip

\begin{table}\label{tab}

  \begin{tabular}{ |p{1cm}||p{5cm}|p{5cm}|p{2cm}|  }
 \hline
 \multicolumn{4}{|c|}{List of  elements in $\mathcal T_\ell $} \\
 \hline
 Value of $\ell$& bosonic &fermionic&degrees \\
 \hline\\
$\ell =3 \dim=1$ & $ t(x_1x_2) $  &$t(x_1^3)$&   2,\ 3\\
$\ell =4$$\dim=1 $&   $t(x_1x_3)$  & $t(x_1^2x_2)$ &2,\ 3\\
 $\ell =5$ $\dim= 3$& $t(x_1x_4)$,$t(x_1^3x_2)$,   $t(x_2x_3)$    & $t(x_1^2x_3)$, $t(x_1 x_2^2)$,   $t(x_1^5)$
  &2,4,2;\ 3,3,5\\\\
  $\ell =6$  $\dim=6$ &$t(x_1x_5)$,    $t(x_1 x_2)^2 $,  $t(x_1^2x_2^2)$,  $t((x_1 x_2)^2)$,     $t(x_1^3x_3)$,   $t(x_2x_4)$  & $t(x_1^2x_4)$,    $t(x_1^4x_2)$, $t(x_1x_2x_3)$, $t(x_1x_3x_2)$,  $t(x_2^3)$, $t(x_1^3) t(x_1x_2)$ & 2,4,4,4,4,2;\ 3,5,3,3,3,5\\\\

  $\ell =7$  $\dim=11$ &$t(x_1x_6)$,   $t(x_2 x_5) $,  $t(x_3x_4)$,    $t(x_1^3x_4)$, $t(x_1 x_2)t(x_1 x_3) $, $t(x_1^2x_3x_2)$, $t(x_1^2x_2x_3) $,  $t(x_1 x_2x_1x_3)$,  $t(x_1 x_2^3) $,              $t(x_1^5x_2)$,  $t(x_1^3) t(x_1^2x_2)$ 
  
   & $t(x_1^2x_5)$,   $t(x_1x_2x_4)$,  $t(x_1x_4x_2)$,  $t(x_1x_3^2)$,  $t( x_2^2x_3)$, $t(x_1^4x_3)$, $t(x_1^3x_2^2)$, $t(x_1^3) t(x_1x_3)$,   $t(x_1^2x_2x_1x_2)$, $t(x_1^2x_2)t(x_1x_2)$, $t(x_1^7)$  & $2^3,4^6,6^2$;\ $3^5, 5^5,7$\\
  
 \hline
\end{tabular}
\end{table}
\subsubsection{Derivatives}
We have rematked from the beginning that in the free algebra  one can substitute fermionic variables only with fermionic elements (the same for bosonic) but we can make a sort of substitution with bosonic variables  developing the notion of derivative  $\pd{}{x_i}$.\medskip

We procede as follows, add a Grassmann variable $\epsilon$  which anticommutes with all the $x_i$  then  by the abstract properties of free algebras given an element $f(x_1,\cdots,x_i,\cdots)$  one can make the substitution $x_i\mapsto x_1+\epsilon g$  in $f$   with $g$ bosonic and then set\begin{equation}\label{der}
f(x_1,\cdots,x_i+\epsilon g,\cdots)=f(x_1,\cdots,x_i,\cdots)+\epsilon g\pd{}{x_i}f(x_1,\cdots,x_i,\cdots)\,.
\end{equation} The derivative  $\pd{}{x_i}$ decreases the degree by 1 and the charge by $i$.

In other words it has  degree -1 and    charge  $-i$.  If $g$ has degree $2h$  and charge  $i$ one has that the operator $g\pd{}{x_i}$ has degree $2h-1$  and charge 0.\smallskip

These operators thus preserve the charge and exchange between bosonic and fermionic elements.\medskip

Multiplication by $x_j,\ j\neq i$  anticommutes with $\pd{}{x_i}$ while 
$$\pd{}{x_i}x_if(x)=f(x)-x_i \pd{}{x_i} f(x),\ [\pd{}{x_i},\pd{}{x_j}]_s=0,\ \forall i\neq j $$ In other words $[\pd{}{x_i},x_i]_s=1$.

Furthermore 
$$\pd{}{x_i}(a\cdot b)=\pd{}{x_i}(a)\cdot b+(-1)^{d(a)}a\cdot \pd{}{x_i}(b);\quad \pd{}{x_i}x_i^{2k}=0,\  \pd{}{x_i}x_i^{2k+1}= x_i^{2k} $$ \medskip

One can the consider a general operator  product of the operators multiplication by elements of the free algebra   and derivatives in particylar we have  the operators of charge  0  which can be either bosonic or fermionic.  Example of fermionic operators of charge 0.
$$x_1^2\pd{}{x_2},\quad  x_1x_2\pd{}{x_3},\quad  x_2\pd{^2}{x_1},\quad  x_4\pd{^2}{x_2}, \quad x_1x_3\pd{}{x_4},\quad   x_2^2\pd{}{x_4},\quad   \quad x_1x_4\pd{}{x_5},$$

$$x_1^3x_2\pd{}{x_5},\quad  x_1x_2^2\pd{}{x_5},\quad  x_1^3x_2\pd{ }{x_1}\pd{ }{x_4},\quad  x_1^3x_2\pd{ }{x_2}\pd{ }{x_3}, \quad x_2x_3\pd{}{x_5},\quad   x_2^2\pd{}{x_5},\quad  $$

\begin{table}\label{tab}

  \begin{tabular}{ |p{1cm}||p{5cm}|p{6cm}|p{2cm}|  }
 \hline
 \multicolumn{4}{|c|}{List of  elements in $\mathcal T_\ell $} \\
 \hline
 Value of $\ell$& bosonic &fermionic&degrees \\
 \hline\\
$\ell =3 \dim=1$ & $ t(x_1x_2) $  &$t(x_1^3)=x_1^2\pd{}{x_2}(t(x_1x_2))$&   2,\ 3\\
$\ell =4$$\dim=1 $&   $t(x_1x_3)$  & $t(x_1^2x_2)=x_1x_2\pd{}{x_3}(t(x_1x_3))$ &2,\ 3\\
 $\ell =5$ $\dim= 3$& $t(x_1x_4)$,$t(x_1^3x_2)$,   $t(x_2x_3)$    & $t(x_1^2x_3)=x_1x_3\pd{}{x_4}(t(x_1x_4))$, $t(x_1 x_2^2)= x_1x_2\pd{}{x_3}(t(x_2x_3))$,   $t(x_1^5)=x_1^2\pd{}{x_2}(t(x_1^3x_2))$
  &2,4,2;\ 3,3,5\\\\
  $\ell =6$  $\dim=6$ &$t(x_1x_5)$,    $t(x_1 x_2)^2 $,  $t(x_1^2x_2^2)$,  $t((x_1 x_2)^2)$,     $t(x_1^3x_3)$,   $t(x_2x_4)$  & $t(x_1^2x_4)=x_1x_4\pd{}{x_5}(t(x_1x_5))$,  $2t(x_1^3) t(x_1x_2)=x_1^2\pd{}{x_2} t(x_1 x_2)^2$  $t(x_1^4x_2)= x_1x_2\pd{}{x_3}(t(x_1^3x_3))$, $t(x_1x_2x_3)=x_3x_1\pd{}{x_4}(t(x_2x_4))=\frac14 \pd{ }{x_2}\circ  x_3\pd{ }{x_1}t((x_1 x_2)^2)$, $t(x_1x_3x_2)=x_1x_3\pd{}{x_4}(t(x_2x_4))$,  $2t(x_2^3)=\pd{ }{x_1}\circ  x_2\pd{ }{x_1}t(x_1^2x_2^2)$,   & 2,4,4,4,4,2;\ 3,5,3,3,3,5\\\\

  $\ell =7$  $\dim=11$ &$t(x_1x_6)$,   $t(x_2 x_5) $,  $t(x_3x_4)$,    $t(x_1^3x_4)$, $t(x_1 x_2)t(x_1 x_3) $, $t(x_1^2x_3x_2)$, $t(x_1^2x_2x_3) $,  $t(x_1 x_2x_1x_3)$,  $t(x_1 x_2^3) $,              $t(x_1^5x_2)$,   
  
   & $t(x_1^2x_5)$,   $t(x_1x_2x_4)$,  $t(x_1x_4x_2)$,  $t(x_1x_3^2)$,  $t( x_2^2x_3)$, $t(x_1^4x_3)$, $t(x_1^3x_2^2)$, $t(x_1^3) t(x_1x_3)$,   $t(x_1^2x_2x_1x_2)$, $t(x_1^2x_2)t(x_1x_2)$, $t(x_1^7)$  & $2^3,4^6,6^2$;\ $3^5, 5^5,7$\\
  
 \hline
\end{tabular}
\end{table}

\eject

 \section{Super \ch \ algebras}  The theory follows verbatim the classical thory but we repeat it for convenience of the reader.\begin{definition}\label{chal}
A super $n$-- \ch\ algebra $R$ is a trace superalgebra $R$ satisfying $t(1)=n$  and all the  identities $CH_{e,f}(x,y) $ of Formula \eqref{leeqb}.
\end{definition}

By   Lemma \ref{labi}  and Formulas \eqref{leeqb1} its trace algebra satisfies the identities $T_{e,f}(x,y) $ of Formula \eqref{leeqb}.\medskip

The category of       super $n$-- \ch\ algebras  admits free algebras and an interpretation of the SFT is
\begin{theorem}\label{SFT}
 The superalgebra $\mathcal E_n\langle X,Y\rangle$ of equivariant polynomial maps from $I$--tuples of $n\times n$ bosonic and $J$--tuples of $n\times n$ fermionic matrices to  $n\times n$ matrices,  is the free algebra in the  category of  $n$--\ch--superalgebras.
 
It is the quotient of the free algebra $\mathcal F_T\langle X,Y\rangle$ modulo the $T$--ideal generated by the $n^{th}$ Cayley Hamilton polynomials of Formula  \eqref{cah} and $t(1)=n$.
  \begin{equation}\label{ffs}
{\boxed{ \mathcal  E_{ n}\langle X,Y\rangle:=\mathcal  F_T\langle X,Y\rangle/\langle CH_{e,f}(x,y),\ e+f=n+1;\ t(1)=n  \rangle}}.
\end{equation}

\end{theorem}

\subsection{$n$--dimensional superrepresentations\label{fdr}} For a given $n\in\mathbb N$ and a supercommutative ring $A$  by $M_n(A)$ we denote the ring of $n\times n$ matrices with coefficients in $A$. By a symbol $(a_{i,j})$  we denote a matrix with entries $ a_{i,j}\in A,$ for $                   i,j=1,\ldots,n    $.  We may distinguish matrices by the $\Z/(2)$ degree.\smallskip

The construction $A\mapsto M_n(A)$ is a functor from the category $\mathcal C$ of commutative superalgebras to the category $\mathcal R$ of associative superalgebras. 

To a map $f:A\to B$ is associated a  map $M_n(f):M_n(A)\to M_n(B)$  in the obvious way $M_n(f)((a_{i,j})):= (f(a_{i,j}))$.
\begin{definition}\label{ndi}
By an  {\em $n$--dimensional superrepresentation}  of a superalgebra $R$ we mean a homomorphism $f:R\to M_n(A)$ of superalgebras with $A$ supercommutative.
\end{definition}
It is easily seen that the set valued functor $A\mapsto \hom_{\mathcal R}(R, M_n(A))$ is representable. That is:
\begin{proposition}\label{laTR}
There is a commutative superalgebra $T_n(R)$  and a natural isomorphism $\mathrm j_A $ of functors:
$$ \hom_{\mathcal R}(R, M_n(A))\stackrel{\mathrm j_A }\simeq  \hom_{\mathcal C}(T_n(R),  A  ),\quad \mathrm j_A :f\mapsto \bar f$$ given by the commutative diagram $f=M_n(\bar f)\circ \mathtt j_R$:
\begin{equation}\label{unpp}
\xymatrix{  R\ar@{->}[r]^{\mathtt j_R\qquad}\ar@{->}[rd]_f&M_n(T_n(R))\ar@{->}[d]^{M_n(\bar f)}\\
&M_n(A) } .
\end{equation} 
\end{proposition} The map $\mathrm j_R:R\to M_n(T_n(R))$ is called the {\em universal $n$--dimensional superrepresentation of $R$ } or {  \em the universal map into $n\times n$ matrices}.

Of course it is possible that $R$ has no   $n$--dimensional superrepresentations, in which case $T_n(R)=\{0\}$. \medskip

\begin{remark}\label{rinal} A completely tautological statement is that $R$ can be embedded  in the ring $M_n(A)$ of $n\times n$ matrices over a supercommutative ring $A$ if and only if  $j_R$ is an embedding.
 
  \end{remark}
The projective linear group $\mathtt P_n=GL(n,\C)/\C^*$  acts by conjugation  of superrepresentations and as in the classical case  it acts on the universal object $M_n(T_n(R))$  and   $R$ maps to the $\mathtt P_n$ invariants.

\begin{theorem}\label{str} If $R$ is an $n$--\ch\ algebra, then the universal map
$\mathtt j_R:R\to M_n(T_n(R))^{\mathtt P_n}$ is an isomorphism inducing an isomorphism of the trace superalgebra of $R$  with $T_n(R) ^{\mathtt P_n}$.
\end{theorem}
\begin{proof}

 For the construction of the universal map we present the superalgebra $R$ as a quotient of the free $n$--\ch \ algebra.
 
 We have thus   $R= F_{T,n}\langle X,Y\rangle/I$ and a commutative diagram:

\begin{equation}\label{unpp0}
\xymatrix{ F_{T,n}\langle X,Y\rangle\ar@{->}[r]^{j}\ar@{->}[d]_{\pi}&M_n(\mathbb K[t(M)])\ar@{->}[d]^{\bar\pi}\\
R\ar@{->}[r]^{\bar j} &M_n(\mathbb K[t(M))/\bar J) } .
\end{equation} where $J=M_n( \bar J)$ is the two sided ideal of  $M_n(\mathbb K[t(M)])$  generated by $I$, then  $T_n(R)=\mathbb K[t(M))/\bar J)$.\medskip

So  
 the question is to prove that   $F_{T,n}\langle X,Y\rangle\cap  M_n( \bar J)= M_n( \bar J)^{\mathtt P_n}=I . $  \smallskip

  In order to simplify notations let us denote $$B:= M_n(\mathbb K[t(M)])),\ J:=M_n( \bar J)$$ we may also assume that the set of variables in the free algebra is infinite.    \smallskip

First $J=\sum_j Bu_j  B ,\ u_j\in I$.   Each $Bu_j  B$  is a  $\mathtt P_n$ submodule and since the group is linerarly reductive  we have $J^{\mathtt P_n}=\sum_j (Bu_j  B)^{ \mathtt P_n} ,\ u_j\in I$.  

So it is enough to prove that for every $u\in I$ of some degree  we have $(Bu  B)^{ \mathtt P_n}\subset I$.\smallskip

Take an element  $c=\sum_ha_h ub_h\in (Bu  B)^{ \mathtt P_n}$ with $a_h,b_h$ homogeneous.

As in  \cite{P5} we then use a matrix variable $x$ with entries not appearing  in $ a_h,b_h,u $  and consider the element 
   $$tr(cx)=tr(\sum_ha_h ub_hx )=tr(\sum_h\pm b_hxa_h u )\in \mathbb K[t(M)]^{ \mathtt P_n}.$$  \smallskip
   
   Since $ \mathtt P_n$ is linearly reductive    we have the Reynolds operator $\mathfrak R$  which projects to the invariants, commutes with trace and satisfies the identity $\mathfrak R(r u)=\mathfrak R(r )u,\ \forall r\in B,\ u\in B^{ \mathtt P_n}$ so \begin{equation}\label{reo}
tr(cx)=    tr(\sum_h\mathfrak R(a_h ub_h x))=  tr(\sum_h\pm\mathfrak R(b_hxa_h)u  ).
\end{equation}
   From this  one can easily prove the main Theorem. \smallskip

    In fact  now  we have the invariant $  \sum_h\pm\mathfrak R(b_hxa_h)$ which is linear in $x$  so by the FFT \ref{FFT} it is of the form $\sum_h A_hxB_h +\sum_k C_k tr(d_kx)$, with $A_h, B_h, C_k, d_k$ invariants independent of $x$.
   
   $$tr(cx)=  tr( \sum_h A_hxB_hu+\sum_k C_k tr(d_kx)u)  = $$
   $$= tr( \sum_h B_huA_hx)+tr(\sum_k u C_k tr(d_kx) ) $$
   $$= tr( \sum_h B_huA_hx+\sum_k  tr(u C_k)d_kx  ) $$\begin{equation}\label{cala}
\implies c=\sum_h B_huA_h +\sum_k  tr(u C_k)d_k\in I
\end{equation} since $I$ is an ideal closed under trace.\medskip

\end{proof}    
\end{document}